\documentclass[11pt,reqno]{amsart}

\usepackage{amsmath,color}
\usepackage{amsthm, enumerate}
\usepackage{amssymb}
\usepackage[all]{xy}
\usepackage{epsfig}

\DeclareMathAlphabet{\mathpzc}{OT1}{pzc}{m}{it}

\newtheorem{theorem}{Theorem}[section]
\newtheorem{proposition}[theorem]{Proposition}
\newtheorem{corollary}[theorem]{Corollary}
\newtheorem{lemma}[theorem]{Lemma}

\theoremstyle{definition}
\newtheorem{definition}[theorem]{Definition}
\newtheorem{example}[theorem]{Example}
\newtheorem{remark}[theorem]{Remark}

\title[Building ideals of two-Lipschitz operators]
{Building ideals of two-Lipschitz operators between mertic and Banach spaces}

\author[D. Achour, E. Dahia]{Dahmane Achour, Elhadj Dahia}

\address{D. Achour\\
	Laboratoire d'Analyse Fonctionnelle et G\'{e}om\'{e}trie des Espaces, University of M'sila \\
	28000 M'sila,
	Algeria.}

\email{dahmane.achour@univ-msila.dz}

\address{Elhadj  Dahia\\
	Laboratoire de Math\'{e}matiques et Physique Appliqu\'{e}es, \'{E}cole
	Normale Sup\'{e}rieure de Bousaada, 28001 Bousaada, Algeria. And Laboratoire d'Analyse Fonctionnelle et G\'{e}om\'{e}trie des Espaces, University of M'sila\\
	28000 M'sila,
	Algeria.}

\email{dahia.elhadj@ens-bousaada.dz}

\date{\today}

\begin{document}
\begin{abstract}
 In this paper, we present and characterize the injective hull of a two-Lipschitz operator ideals and the definition of two-Lipschitz dual operator ideal. Also we introduce two methods for
 creating ideals of two-Lipschitz operators from a pair of Lipschitz operator
 ideals. Namely, Lipschitzization and factorization method. We show the
 closedness, the injectivity and the symmetry of these two-Lipschitz ideals
 according to the closedness, injectivity and symmetry of the corresponding
 Lipschitz operator ideals. Some illustrative examples are given.

 
 
\end{abstract}
\maketitle

\section{Introduction}
The concept of operator ideals has emerged as a powerful method for
exploring and categorizing linear operators that operate between Banach
spaces. The principles of linear theory have extended to Lipschitz
operators, giving rise to the concepts of Lipschitz operator ideals. Farmer
and Johnson provided the initial framework for a Lipschitz theory in 2009
(see \cite{FG09}). Achour et al., in \cite{ARSY16}, presented an axiomatic
theory for Banach spaces-valued Lipschitz mappings in the year 2016 (see
also \cite{inj} for Lipschitz operator ideals between pointed metric
spaces). It should be noted that in recent years, various authors have
successfully extended several operator ideals to the Lipschitz setting (see 
\cite{YrAdPr}, \cite{APY2017}, \cite{ADS2018}, \cite{pcom}, \cite{AM}, \cite
{CCJV16}, \cite{Dahia1}, \cite{Tiaiba}, and the references therein).

The definition of real two-Lipschitz maps acting in a cartesian product of
two pointed metric spaces was introduced by S\'{a}nchez P\'{e}rez in \cite
{En15}. A detailed and systematic study of these mappings with values in a
Banach space is given recently by Hamidi et al. in \cite{Hamidietall}, where
the authors introduce the new concept of two-Lipschitz operator ideals
between pointed metric spaces and Banach space. These ideals constitute
classes of two-Lipschitz operators that exhibit stability when subjected to
the composition of a linear operator on the right and two Lipschitzian
operators on the left. This attempt marked an initial effort to extend the
theory of ideals of bilinear operators into the two-Lipschitz settings. A
numerous notable examples and further results, are detailed in \cite{Hamidietall} and \cite{Dahia}. Note that the
notion of two-Lipschitz mappings was initially introduced by Dubei et al. 
\cite{MEA} as those mappings that are Lipschitz in each variable.
The aim of this work is to defined some procedures which assign new
two-Lipschitz operator ideals from a given one or from a pair of Lipschitz
operator ideals. The essence is that, depending on the Lipschitz operator
ideals, it is possible that there is not a unique natural generalization to
the two-Lipschitz contexts. Actually, these procedures are inspired from the
general procedures to construct ideals of multilinear operators starting
from a given linear operator ideal, was introduced by Pietsch in \cite{99}
(see also \cite{9}).

Our results are presented as follows. After the introductory one, in Section 2 we fix notation and we recall the most important results on the theory of
Lipschitz operators and two-Lipschitz operators that we will use throughout
the manuscript. Section 3 is devoted to the closedness and injectivity of a
given two-Lipschitz operator ideals. After showing that the closure of a
two-Lipschitz operator ideal is a closed two-Lipschitz ideal, we introduce
and characterize the two-Lipschitz injective hull of two-Lipschitz operator
ideals defined between cartesian product of two pointed metric spaces and
Banach space. In Section 4, we define the two-Lipschitz dual of a given operator ideal, obtaining in this way an ideal of two-Lipschitz operators created by the composition method. In the last section, which is the core of the article, we present general techniques to generate two-Lipschitz operator ideals from given Lipschitz operator ideals. That is Lipschitzization method,
factorization method and strong factorization method. We show the
closedness, the injectivity and the symmetry of these two-Lipschitz ideals
according to closedness, injectivity and symmetry of the corresponding
Lipschitz operator ideals.

\section{Notation and preliminaries}
From now on, unless otherwise stated, $X,X_{1},X_{2},W,Y$ and $Z$ will
denote pointed metric spaces with base point $0$ that is, $0$ is any
arbitrary fixed point of $X$ and metric will be denoted by $d$. We denote by 
$B_{X}=\left\{ x\in X:d(x,0)\leq 1\right\} $. Also, $E,F,G_{1},G_{2}$ and $G$
denote Banach spaces over the same field $\mathbb{K}$ (either $\mathbb{R}$
or $\mathbb{C}$) with dual spaces $E^{\ast }$. A Banach space $E$ is a
pointed metric space with distinguished point $0$ (the null vector) and
metric $d(x,x^{\prime })=\Vert x-x^{\prime }\Vert $. With $Lip_{0}(X,Y)$ we
denote the set of all Lipschitz mappings from $X$ to $Y$ such that $T(0)=0$.
In particular $Lip_{0}(X,E)$ is a Banach space under the Lipschitz norm%
\begin{equation*}
Lip(T)=\inf \left\{ C>0:\left\Vert T(x)-T(x^{\prime })\right\Vert \leq
Cd\left( x,x^{\prime }\right) ;\forall x,x^{\prime }\in X\right\} .
\end{equation*}%
The space $Lip_{0}(X,\mathbb{K})$ is called the Lipschitz dual of $X$ and it
will be denoted by $X^{\#}$. It is worth mentioning that the Banach space of
all linear operators $\mathcal{L}(E,F)$ is a subspace of $Lip_{0}(E,F)$ and,
so, $E^{\ast }$ is a subspace of $E^{\#}$. A molecule on $X$ is a finitely
supported function $m:X\longrightarrow \mathbb{R}$ that satisfies $
\sum\limits_{x\in X}m\left( x\right) =0$. The set $\mathcal{M}(X)$ of all
molecules is a vector space, those of the form $m=\sum_{j=1}^{n}\lambda
_{j}m_{x_{j}x_{j}^{\prime }},$ where $m_{xx^{\prime }}=\chi _{\left\{
	x\right\} }-\chi _{\left\{ x^{\prime }\right\} }$ with $\lambda _{j}\in 
\mathbb{R}$ and $x_{j},x_{j}^{\prime }\in X$ and $\chi _{A}$ is the
characteristic function of the set $A$. The completion of $\mathcal{M}(X)$,
when endowed with the norm $\left\Vert m\right\Vert _{\mathcal{M}%
	(X)}=\sum_{j=1}^{n}\left\vert \lambda _{j}\right\vert d(x_{j},x_{j}^{\prime
}),$ where the infimum is taken over all representations $%
\sum_{j=1}^{n}\lambda _{j}m_{x_{j}x_{j}^{\prime }}$ of the molecule $m,$ is
denoted \AE $\left( X\right) $ and called Arens--Ells space associated to $X$
(see \cite{A.E56}). Consider the canonical Lipschitz injection map $\delta
_{X}:X\longrightarrow $\AE $(X)$ defined by $\delta _{X}(x)=m_{x0}$
isometrically embeds $X$ in \AE $\left( X\right) $. If we consider $T\in
Lip_{0}(X,E)$, then $T$ always factors through \AE $(X)$ as $T=T_{L}\circ
\delta _{X}:X\longrightarrow $\AE $(X)\longrightarrow E.$ The map $T_{L}:$ 
\AE $(X)\longrightarrow E$ is the unique continuous linear operator
(referred to as the linearization of $T$) satisfies $T=T_{L}\circ \delta _{X}
$ and $\Vert T_{L}\Vert =Lip(T)$ (see \cite[Theorem 3.6]{Weav2}).

The notion of linear operator ideal extends (by Achour et al. in \cite{ARSY16}) to
Lipschitz operators as follows.

\begin{definition}
A Lipschitz operator ideal $\mathcal{I}_{{L}ip}$ is a subclass of $Lip_{0}$ such that for every pointed
metric space $X$ and every Banach space $E$ the components
$$
\mathcal{I}_{{L}ip}(X,E):=Lip_{0}(X,E)\cap\mathcal{I}_{{L}ip}
$$
satisfy
\begin{enumerate} [\upshape (i)]
	\item $\mathcal{I}_{{L}ip}(X,E)$ is a linear subspace of $Lip_{0}(X,E)$.
	\item $vg\in\mathcal{I}_{{L}ip}(X,E)$ for $v\in E$ and $g\in X^{\#}$.
	\item The ideal property: if $S\in Lip_{0}(Y,X)$, $T\in
	\mathcal{I}_{{L}ip}(X,E)$ and $w\in\mathcal{L}(E,F)$, then the composition
	$wTS$ is in $\mathcal{I}_{{L}ip}(Y,F)$.
\end{enumerate}
A Lipschitz operator ideal $\mathcal{I}_{{L}ip}$\ is a normed \ (Banach) Lipschitz operator ideal if
there is $\left\Vert .\right\Vert _{\mathcal{I}_{{L}ip}}:\mathcal{I}_{{L}
	ip}\longrightarrow\left[  0,+\infty\right[  $ that satisfies
\begin{enumerate}[\upshape (i')]
	\item For every pointed metric space $X$ and every Banach
	space $E$, the pair $(\mathcal{I}_{{L}ip}(X,E),\left\Vert .\right\Vert _{\mathcal{I}_{{L}ip}})$ is a normed (Banach) space and $Lip(T)\leq\Vert
	T\Vert_{\mathcal{I}_{{L}ip}}$ for all $T\in\mathcal{I}_{{L}ip}(X,E)$.
	\item $\left\Vert Id_{\mathbb{K}}:\mathbb{K} \longrightarrow\mathbb{K},Id_{\mathbb{K}}(\lambda)=\lambda\right\Vert
	_{\mathcal{I}_{{L}ip}}=1$.
	\item If $S\in Lip_{0}(Y,X)$, $T\in\mathcal{I}_{{L}ip}(X,E)$ and $w\in\mathcal{L}(E,F),$ then
	$\left\Vert wTS\right\Vert _{\mathcal{I}_{{L}ip}}\leq Lip(S)\left\Vert
	T\right\Vert _{\mathcal{I}_{{L}ip}}\left\Vert w\right\Vert .$
\end{enumerate}
\end{definition}

Regarding the two-Lipschitz operators, according to \cite[Definition 2.1]{Hamidietall}, a map $T:X_{1}\times
X_{2}\longrightarrow E$, from the cartesian product of two pointed metric
spaces $X_{1}$ and $X_{2}$ to a Banach space $E$, is called two-Lipschitz
operator if there is a constant $C>0$ such that for every pair of element $%
x_{1},x_{1}^{\prime }\in X_{1}$ and $x_{2},x_{2}^{\prime }\in X_{2},$

\begin{equation}
\left\Vert T\left( x_{1},x_{2}\right) -T\left( x_{1},x_{2}^{\prime }\right)
-T\left( x_{1}^{\prime },x_{2}\right) +T\left( x_{1}^{\prime },x_{2}^{\prime
}\right) \right\Vert \leq C.d\left( x_{1},x_{1}^{\prime }\right) d\left(
x_{2},x_{2}^{\prime }\right)   \label{defbilip}
\end{equation}
and $T(x_{1},0)=T(0,x_{2})=0.$ We will denote by $BLip_{0}\left( X,Y;E\right) $ the Banach space of all these mappings, under the norm
\begin{equation}
BLip\left( T\right) =\sup_{x_{1}\neq x_{1}^{\prime },x_{2}\neq x_{2}^{\prime
}}\frac{\left\Vert T\left( x_{1},x_{2}\right) -T\left( x_{1},x_{2}^{\prime
	}\right) -T\left( x_{1}^{\prime },x_{2}\right) +T\left( x_{1}^{\prime
	},x_{2}^{\prime }\right) \right\Vert }{d\left( x_{1},x_{1}^{\prime }\right)
	d\left( x_{2},x_{2}^{\prime }\right) }. \label{defbilip1}
\end{equation}

Recently, the notion of two-Lipschitz operators ideals was introduced by the
author, Hamidi et al. \cite{Hamidietall}.

\begin{definition}
	\label{defideal} A two-Lipschitz operator ideal between pointed metric
	spaces and Banach spaces, $\mathcal{I}_{BLip}$, is a subclass of $BLip_{0}$
	such that for every pointed metric spaces $X_{1},X_{2}$ and every Banach
	space $E$ the components 
	\begin{equation*}
	\mathcal{I}_{BLip}(X_{1},X_{2};E):=BLip_{0}(X_{1},X_{2};E)\cap \mathcal{I}%
	_{BLip}
	\end{equation*}
	satisfy
	
	\begin{enumerate}
		\item[(i)] $\mathcal{I}_{BLip}(X_{1},X_{2};E)$ is a linear subspace of $%
		BLip_{0}(X_{1},X_{2};E)$.
		
		\item[(ii)] For any $f_{1}\in X_{1}^{\#},$ $f_{2}\in X_{2}^{\#}$ and $e\in E$%
		, the map $f_{1}\cdot f_{2}\cdot e$ belongs to $\mathcal{I}%
		_{BLip}(X_{1},X_{2};E)$.
		
		\item[(iii)] The ideal property: if $f_{1}\in Lip_{0}(Z,X_{1})$, $f_{2}\in
		Lip_{0}(W,X_{2})$ , \newline $T\in \mathcal{I}_{BLip}(X_{1},X_{2};E)$ and $u\in 
		\mathcal{L}(E,F)$, then the composition $u\circ T\circ (f_{1},f_{2})$ is in $%
		\mathcal{I}_{BLip}(Z,W;F)$.
	\end{enumerate}
	
	A two-Lipschitz operator ideal $\mathcal{I}_{BLip}$ is a normed (Banach)
	two-Lipschitz operator ideal if there is $\Vert \cdot \Vert _{\mathcal{I}
		_{BLip}}:\mathcal{I}_{BLip}\longrightarrow \left[ 0,+\infty \right[ $ that
	satisfies
	
	\begin{enumerate}
		\item[(i')] For every pointed metric spaces $X,Y$ and every Banach space $E$
		, the pair $(\mathcal{I}_{BLip}(X_{1},X_{2};E),\Vert \cdot \Vert _{\mathcal{I%
			}_{BLip}})$ is a normed (Banach) space and $BLip(T)\leq \Vert T\Vert _{%
			\mathcal{I}_{B{L}ip}}$ for all $T\in \mathcal{I}_{BLip}(X_{1},X_{2};E)$.
		
		\item[(ii')] $\left\Vert Id_{\mathbb{K}^{2}}:\mathbb{K}\times \mathbb{K}%
		\longrightarrow \mathbb{K}:Id_{\mathbb{K}^{2}}\left( \alpha ,\beta \right)
		=\alpha \beta \right\Vert _{\mathcal{I}_{BLip}}=1$.
		
		\item[(iii')] If $f_{1}\in Lip_{0}(Z,X_{1})$, $f_{2}\in Lip_{0}(W,X_{2})$, $%
		T\in \mathcal{I}_{BLip}(X_{1},X_{2};E)$ and $u\in \mathcal{L}(E,F)$, the
		inequality $\Vert u\circ T\circ (f_{1},f_{2})\Vert _{\mathcal{I}_{BLip}}\leq
		\left\Vert u\right\Vert \Vert T\Vert _{\mathcal{I}%
			_{BLip}}Lip(f_{1})Lip(f_{2})$ holds.
	\end{enumerate}
\end{definition}
Note that if we assume that $\left\Vert \cdot \right\Vert _{BLip_{0}}$ is a
quasi-norm instead of a norm, we will say that $\left( \mathcal{I}
_{BLip},\left\Vert \cdot \right\Vert _{\mathcal{I}_{BLip}}\right) $ is a
quasi-normed two-Lipschitz operator ideal.

Now we recall the composition method, introduced in \cite{Hamidietall}, in
order to produce two-Lipschitz operator ideal $\mathcal{I\circ }BLip_{0}$
from a given linear operator ideal $\mathcal{I}$. Indeed, the two-Lipschitz
operator $T:X_{1}\times X_{2}\longrightarrow E$ belongs to $\mathcal{I\circ }
BLip_{0}(X_{1},X_{2};E)$ if there is a Banach space $F$, a two-Lipschitz
operator $S\in BLip_{0}(X_{1},X_{2};E)$ and a linear operator $u\in \mathcal{
	I}(F,E)$ such that $T=u\circ S$. This definition is equivalent to say that $
T_{L}\in \mathcal{I}($\AE $\left( X_{1}\right) \widehat{\otimes }_{\pi }$\AE $
\left( X_{2}\right) ,E)$. If $(\mathcal{I},\left\Vert \cdot \right\Vert _{
	\mathcal{I}})$ is a normed operator ideal we write
\begin{equation}
\left\Vert T\right\Vert _{\mathcal{I\circ }BLip_{0}}=\inf \left\Vert
u\right\Vert _{\mathcal{I}}BLip(S)=\left\Vert T_{L}\right\Vert _{\mathcal{I}
}, \label{compo}
\end{equation}
where the infimum is taken over all $u$, $S$ as above (see \cite[Theorem 3.5]{Hamidietall}).

\section{Closed and injective two-Lipschitz operator ideal}

Before giving the first example of two-Lipschitz operator ideal, let us see
that the closure of a two-Lipschitz operator ideal is a closed two-Lipschitz
operator ideals.
\begin{definition}
	Let $\mathcal{I}_{BLip}$ be a two-Lipschitz operator ideal. It will be
	denoted by $\overline{\mathcal{I}_{BLip}}$ the class of two-Lipschitz
	operators formed by components $\overline{\mathcal{I}_{BLip}}(X_{1},X_{2};E)$
	that are given by the closure of $\mathcal{I}_{BLip}(X_{1},X_{2};E)$ in $%
	BLip_{0}(X_{1},X_{2};E).$ The two-Lipschitz operator ideal$\ \mathcal{I}%
	_{BLip}$\ is said to be closed when $\overline{\mathcal{I}_{BLip}}=\mathcal{I%
	}_{BLip}.$
\end{definition}

\begin{proposition}
	If $\mathcal{I}_{BLip}$ is a two-Lipschitz operator ideal, then the class $%
	\overline{\mathcal{I}_{BLip}}$ defined by 
	\begin{equation*}
	\overline{\mathcal{I}_{BLip}}(X_{1},X_{2};E):=\overline{\mathcal{I}%
		_{BLip}(X_{1},X_{2};E)}^{BLip(\cdot )},
	\end{equation*}%
	for all pointed metric spaces $X_{1},X_{2}$ and Banach space $E,$ is a
	two-Lipschitz operator ideal.
\end{proposition}

\begin{proof}
	We check only the condition (iii) in the Definition \ref{defideal}. Let $%
	f_{1}\in Lip_{0}(Z,X_{1})$, $f_{2}\in Lip_{0}(W,X_{2})$ , $T\in \overline{%
		\mathcal{I}_{BLip}}(X_{1},X_{2};E)$ and $u\in \mathcal{L}(E,F)$. Then there
	exists a sequence $(T_{n})_{n}$ in $\mathcal{I}_{BLip}(X_{1},X_{2};E)$ such
	that $\underset{n\longrightarrow +\infty }{\lim }BLip(T_{n}-T)=0$. Now, in $%
	\mathcal{I}_{BLip}(X_{1},X_{2};E),$ consider the sequence $\left( u\circ
	T_{n}\circ (f_{1},f_{2})\right) _{n}.$ By using simple calculation, we get%
	\begin{eqnarray*}
		&&BLip\left( u\circ T_{n}\circ (f_{1},f_{2})-u\circ T\circ
		(f_{1},f_{2})\right)  \\
		&\leq &\left\Vert u\right\Vert BLip(T_{n}-T)Lip(f_{1})Lip(f_{2}),
	\end{eqnarray*}%
	showing that the sequence $\left( u\circ T_{n}\circ (f_{1},f_{2})\right) _{n}
	$ is convergent to $u\circ T\circ (f_{1},f_{2})$, and the proof follows.
\end{proof}
\begin{example}
	A two-Lipschitz operator $T:X_{1}\times X_{2}\longrightarrow E$ is said to
	be approximable, in symbols $T\in \overline{BLip_{0,\mathcal{F}}}%
	(X_{1},X_{2};E)$ if it is the limit of a sequence of finite type
	two-Lipschitz operators $(T_{n})_{n}\subset BLip_{0,\mathcal{F}%
	}(X_{1},X_{2};E)$ (see \cite[Definition 2.8]{Hamidietall} )in the two-Lipschitz norm $BLip(\cdot )$. That is%
	\begin{equation*}
	\overline{BLip_{0,\mathcal{F}}}(X_{1},X_{2};E):=\overline{BLip_{0,\mathcal{F}%
		}(X_{1},X_{2};E)}^{BLip(\cdot )}.
	\end{equation*}%
	Let now $\mathcal{I}_{BLip}$ be a given closed two-Lipschitz operator ideal.
	The above proposition and the following inclusions%
	\begin{equation*}
	\overline{BLip_{0,\mathcal{F}}(X_{1},X_{2};E)}\subset \overline{\mathcal{I}%
		_{BLip}(X_{1},X_{2};E)}=\mathcal{I}_{BLip}(X_{1},X_{2};E),
	\end{equation*}
	asserts that the class of approximable two-Lipschitz operators is the
	smallest closed two-Lipschitz operator ideal. In particular, $\left( 
	\overline{BLip_{0,\mathcal{F}}},BLip(\cdot )\right) $ is a Banach
	two-Lipschitz operator ideal.
\end{example}

In what follows, if $E$ is a Banach space, we consider the natural metric
injection $J_{E}:E\longrightarrow \ell _{\infty }(B_{E^{\ast }})$ given by $%
J_{E}(y)=\left( \left\langle y,y^{\ast }\right\rangle \right) _{y^{\ast }\in
	B_{E^{\ast }}}.$ Note that $\left\Vert J_{E}\right\Vert =1$ and by a simple calculation we obtain
\begin{equation}
BLip(T)=BLip(J_{E}\circ T)  \label{binj}.
\end{equation}
We propose a definition of the injective hull of a two-Lipschitz operator
ideal which extend that introduced by Stephani for the linear case in \cite
{Stephani}. Actually the injective hull defined a procedure which assign new Banach
two-Lipschitz operator ideal from a given one.

\begin{definition}
	Let $\mathcal{I}_{BLip}$ be a two-Lipschitz operator ideal. For pointed
	metric spaces $X_{1},X_{2}$ and a Banach space $E$, a two-Lipschitz operator 
	$T\in BLip_{0}(X_{1},X_{2};E)$ belongs to the two-Lipschitz injective hull
	of $\mathcal{I}_{BLip}$ if $J_{E}\circ T\in \mathcal{I}_{BLip}(X_{1},X_{2};
	\ell _{\infty }(B_{E^{\ast }}))$. The class of all two-Lipschitz from $
	X_{1},X_{2}$ to $E$ which belongs to the Lipschitz injective hull of $
	\mathcal{I}_{BLip}$ will be denoted by $\mathcal{I}
	_{BLip}^{inj}(X_{1},X_{2};E).$
\end{definition}

For any two-Lipschitz operator ideal $\mathcal{I}_{BLip}$, we may assign to $
\mathcal{I}_{BLip}^{inj}$ a real valued function $\left\Vert \cdot
\right\Vert _{\mathcal{I}_{BLip}^{inj}}$ defined by
\begin{equation*}
\begin{array}{cc}
\left\Vert T\right\Vert _{\mathcal{I}_{BLip}^{inj}}:=\left\Vert J_{E}\circ
T\right\Vert _{\mathcal{I}_{BLip}}.
\end{array}
\end{equation*}

Now we give a generalization of a well known characterization of the
injective hull of a linear operator ideal, which first appear in \cite%
{Stephani}. Since the proof follows in the same way as in the linear case we
omit it.

\begin{proposition}
	\label{prop1} Let $\left( \mathcal{I}_{BLip},\left\Vert \cdot \right\Vert _{%
		\mathcal{I}_{BLip}}\right) $ and $\left( \mathcal{I}_{BLip}^{\prime
	},\left\Vert \cdot \right\Vert _{\mathcal{I}_{BLip}^{\prime }}\right) $ be
	normed (Banach) two-Lipschitz operator ideals.
	
	\begin{enumerate}
		\item[(1)] $\mathcal{I}_{BLip}\subset \mathcal{I}_{BLip}^{inj}$ and $
		\left\Vert \cdot \right\Vert _{\mathcal{I}_{BLip}^{inj}}\leq \left\Vert 
		\cdot \right\Vert _{\mathcal{I}_{BLip}}$.
		
		\item[(2)] If $\mathcal{I}_{BLip}\subset \mathcal{I}_{BLip}^{\prime }$ and $%
		\left\Vert \cdot \right\Vert _{\mathcal{I}_{BLip}^{\prime }}\leq \left\Vert
		\cdot \right\Vert _{\mathcal{I}_{BLip}}$ then $\mathcal{I}%
		_{BLip}^{inj}\subset (\mathcal{I}_{BLip}^{\prime })^{inj}$ and $\left\Vert
		\cdot \right\Vert _{(\mathcal{I}_{BLip}^{\prime })^{inj}}\leq \left\Vert
		\cdot \right\Vert _{\mathcal{I}_{BLip}^{inj}}.$ 
	\end{enumerate}
\end{proposition}
The main result of this section is the following

\begin{proposition}
	If $\left( \mathcal{I}_{BLip},\left\Vert \cdot \right\Vert _{\mathcal{I}
		_{BLip}}\right) $ is a normed (Banach) two-Lipschitz operator ideal, then $
	\mathcal{I}_{BLip}^{inj}$ together with $\left\Vert \cdot \right\Vert _{
		\mathcal{I}_{BLip}^{inj}}$ is a normed (Banach) two-Lipschitz operator ideal.
\end{proposition}
\begin{proof}
	For every $S,T\in \mathcal{I}_{BLip}^{inj}(X_{1},X_{2};E)$ and $\alpha \in 
	\mathbb{K}$ we have 
	\begin{equation}
	J_{E}\circ (\alpha S+T)=\alpha J_{E}\circ S+J_{E}\circ T\in \mathcal{I}%
	_{BLip}(X_{1},X_{2};E).  \label{fr}
	\end{equation}%
	This formula ensures that $\mathcal{I}_{BLip}^{inj}(X_{1},X_{2};E)$ is a
	linear subspace of $BLip_{0}(X_{1},X_{2};E)$, which from Proposition \ref
	{prop1} (1), contains the two-Lipschitz operators of the form $f_{1}\cdot
	f_{2}\cdot e$, where $f_{1}\in X_{1}^{\#}$, $f_{2}\in X_{2}^{\#}$ and $e\in E
	$. On the other hand,
	\begin{equation*}
	BLip(T)=BLip(J_{E}\circ T)\leq \left\Vert J_{E}\circ T\right\Vert _{\mathcal{
			I}_{BLip}}=\left\Vert T\right\Vert _{\mathcal{I}_{BLip}^{inj}}.
	\end{equation*}%
	By the norm axioms for $\left\Vert \cdot \right\Vert _{\mathcal{I}_{BLip}}$,
	we easily get the norm axioms for $\left\Vert \cdot \right\Vert _{\mathcal{I}
		_{BLip}^{inj}}$. Also we have $Id_{\mathbb{K}^{2}}\in \mathcal{I}
	_{BLip}\subset \mathcal{I}_{BLip}^{inj}$ and by \cite[Proposition 3.2]{Hamidietall}, we obtain
	directly
	\begin{eqnarray*}
		\left\Vert Id_{\mathbb{K}^{2}}\right\Vert _{\mathcal{I}_{BLip}^{inj}}
		&=&\left\Vert J_{\mathbb{K}}\circ Id_{\mathbb{K}^{2}}\right\Vert _{\mathcal{I
			}_{BLip}} \\
		&=&\left\Vert id_{\mathbb{K}}\cdot id_{\mathbb{K}}\cdot J_{\mathbb{K}
		}(1)\right\Vert _{\mathcal{I}_{BLip}} \\
		&=&\left\Vert J_{\mathbb{K}}(1)\right\Vert _{\infty }Lip(id_{\mathbb{K}
		})Lip(id_{\mathbb{K}})=1.
	\end{eqnarray*}
	Now we prove the ideal property, let $f_{1}\in Lip_{0}(Z,X_{1})$, $f_{2}\in
	Lip_{0}(W,X_{2})$, $T\in \mathcal{I}_{BLip}^{inj}(X_{1},X_{2};E)$ and $u\in 
	\mathcal{L}(E,F)$, by the metric extension property of $\ell _{\infty
	}(B_{F^{\ast }})$ (see \cite[Page 33]{18}) there exists a linear operator $
	v\in \mathcal{L}(\ell _{\infty }(B_{E^{\ast }}),\ell _{\infty }(B_{F^{\ast
	}}))$ such that $v\circ J_{E}=J_{F}\circ u$ and $\left\Vert v\right\Vert
	=\left\Vert J_{F}\circ u\right\Vert =\left\Vert u\right\Vert .$ Therefore,
	\begin{equation*}
	J_{F}\circ u\circ T\circ (f_{1},f_{2})=v\circ J_{E}\circ T\circ
	(f_{1},f_{2})\in \mathcal{I}_{BLip}(Z,W;\ell _{\infty }(B_{F^{\ast }})),
	\end{equation*}%
	which means that $u\circ T\circ (f_{1},f_{2})$ is in $\mathcal{I}
	_{BLip}^{inj}(Z,W;F)$. In addition,%
	\begin{eqnarray*}
		\left\Vert u\circ T\circ (f_{1},f_{2})\right\Vert _{\mathcal{I}
			_{BLip}^{inj}} &=&\left\Vert v\circ J_{E}\circ T\circ
		(f_{1},f_{2})\right\Vert _{\mathcal{I}_{BLip}} \\
		&\leq &\left\Vert v\right\Vert \left\Vert J_{E}\circ T\right\Vert _{\mathcal{
				I}_{BLip}}Lip(f_{1})Lip(f_{2}) \\
		&=&\left\Vert u\right\Vert \left\Vert T\right\Vert _{\mathcal{I}
			_{BLip}^{inj}}Lip(f_{1})Lip(f_{2}).
	\end{eqnarray*}%
	In order to prove that $\left( \mathcal{I}_{BLip}^{inj}(X_{1},X_{2};E),\left
	\Vert \cdot \right\Vert _{\mathcal{I}_{BLip}^{inj}}\right) $ is a Banach
	space, we will consider a Cauchy sequence $\left( T_{n}\right) _{n\geq
		0}\subset \mathcal{I}_{BLip}^{inj}(X_{1},X_{2};E).$ Hence for all $
	\varepsilon >0$, there exists $n_{\varepsilon }\in \mathbb{N}$ such that 
	\begin{equation}
	BLip(T_{n}-T_{k})\leq \left\Vert T_{n}-T_{k}\right\Vert _{\mathcal{I}
		_{BLip}^{inj}}=\left\Vert J_{E}\circ T_{n}-J_{E}\circ T_{k}\right\Vert _{
		\mathcal{I}_{BLip}}<\varepsilon ,  \label{dcauc}
	\end{equation}
	for all $n,k\geq n_{\varepsilon }.$ Which means that $\left( T_{n}\right)
	_{n\in \mathbb{N}}$\ is a Cauchy sequence in the Banach space $%
	BLip_{0}(X_{1},X_{2};E).$ Thus, it exists $T\in BLip_{0}(X_{1},X_{2};E)$
	such that $T_{n}\longrightarrow T$ and then $J_{E}\circ T_{n}\longrightarrow
	J_{E}\circ T$ in $BLip_{0}(X_{1},X_{2};\ell _{\infty }(B_{E^{\ast }}))$.
	Also, the inequality (\ref{dcauc}) asserts that $\left( J_{E}\circ
	T_{n}\right) _{n\geq 0}$\ is a Cauchy sequence in the Banach space $\mathcal{%
		I}_{BLip}(X_{1},X_{2};\ell _{\infty }(B_{E^{\ast }}))$. Then there exists $%
	S\in \mathcal{I}_{BLip}(X_{1},X_{2};\ell _{\infty }(B_{E^{\ast }}))$ such
	that $J_{E}\circ T_{n}\longrightarrow S$. By the uniqueness we get $%
	S=J_{E}\circ T\in \mathcal{I}_{BLip}(X_{1},X_{2};\ell _{\infty }(B_{E^{\ast
	}}))$. Finally, it follows from%
	\begin{equation*}
	\left\Vert T_{n}-T\right\Vert _{\mathcal{I}_{BLip}^{inj}}=\left\Vert
	J_{E}\circ T_{n}-J_{E}\circ T\right\Vert _{\mathcal{I}_{BLip}}=\left\Vert
	J_{E}\circ T_{n}-S\right\Vert _{\mathcal{I}_{BLip}}\longrightarrow 0,
	\end{equation*}%
	that $\left( T_{n}\right) _{n\geq 0}$ converges to $T$ in $\mathcal{I}%
	_{BLip}^{inj}(X_{1},X_{2};E)$.
\end{proof}

The next result describe the injective hull of a Banach two-Lipschitz
operator ideal of composition type.

\begin{proposition}
	Let $\mathcal{I}$ be a Banach linear operator ideal. Then 
	\begin{equation}
	(\mathcal{I\circ }BLip_{0})^{inj}=\mathcal{I}^{inj}\mathcal{\circ }BLip_{0},  \label{exinj}
	\end{equation}
	isometrically. Consequently, if $\mathcal{I}_{BLip}$ is a Banach
	two-Lipschitz operator ideal constructed by the composition method, then so
	is also $\mathcal{I}_{BLip}^{inj}$.
\end{proposition}

\begin{proof}
	Firstly, note that the uniqueness of the linearization maps gives that $%
	(J_{E}\mathcal{\circ }T)_{L}=J_{E}\mathcal{\circ }T_{L}$ for all $T\in
	BLip_{0}(X_{1},X_{2};E).$ We have that $T$ belongs to $(\mathcal{I\circ }%
	BLip_{0})^{inj}(X_{1},X_{2};E)$ if and only if $J_{E}\mathcal{\circ }T$
	belongs to $\mathcal{I\circ }BLip_{0}(X_{1},X_{2};\ell _{\infty }(B_{E^{\ast
	}}))$ and this is equivalent to say that its linearization $(J_{E}\mathcal{%
		\circ }T)_{L}=J_{E}\mathcal{\circ }T_{L}$ belongs to $\mathcal{I}($\AE $%
	\left( X\right) \widehat{\otimes }_{\pi }$\AE $\left( Y\right) ,E)$ and this
	is equivalent to $T_{L}\in \mathcal{I}^{inj}($\AE $\left( X_{1}\right) 
	\widehat{\otimes }_{\pi }$\AE $\left( X_{2}\right) ,E)$, and the result follows.
\end{proof}

\begin{definition}
	The normed two-Lipschitz operator ideal $\left( \mathcal{I}
	_{BLip},\left\Vert \cdot \right\Vert _{\mathcal{I}_{BLip}}\right) $ is said
	to be injective if $\mathcal{I}_{BLip}=\mathcal{I}_{BLip}^{inj}$ (i.e. if $
	T\in BLip_{0}(X_{1},X_{2};E)$ is such that $J_{E}\circ T\in \mathcal{I}
	_{BLip}(X_{1},X_{2};\ell _{\infty }(B_{E^{\ast }}))$, then $T\in \mathcal{I}
	_{BLip}(X_{1},X_{2};E$). Moreover $\left\Vert \cdot \right\Vert _{\mathcal{I}
		_{BLip}}=\left\Vert \cdot \right\Vert _{\mathcal{I}_{BLip}^{inj}}$.
\end{definition}
	
\begin{remark}
	\label{injec} By (1) in Proposition \ref{prop1}, $\left( \mathcal{I}
	_{BLip},\left\Vert \cdot \right\Vert _{\mathcal{I}_{BLip}}\right) $ is
	injective if and only if $T\in \mathcal{I}_{BLip}(X_{1},X_{2};E)$ whenever $
	T\in BLip_{0}(X_{1},X_{2};E)$ such that $J_{E}\circ T\in \mathcal{I}
	_{BLip}(X_{1},X_{2};\ell _{\infty }(B_{E^{\ast }}))$ and $\left\Vert
	T\right\Vert _{\mathcal{I}_{BLip}}=\left\Vert J_{E}\circ T\right\Vert _{
		\mathcal{I}_{BLip}}$.
\end{remark}

We aim to characterize injective normed two-Lipschitz operator ideals by
metric injections.

\begin{theorem}
	The normed two-Lipschitz operator ideal $\left( \mathcal{I}%
	_{BLip},\left\Vert \cdot \right\Vert _{\mathcal{I}_{BLip}}\right) $ is
	injective if and only if $T\in \mathcal{I}_{BLip}(X_{1},X_{2};E)$ whenever $
	T\in BLip_{0}(X_{1},X_{2};E)$ and $u\in \mathcal{L}(E,G)$ is a metric
	injection such that $u\circ T\in \mathcal{I}_{BLip}(X_{1},X_{2};E)$ and $
	\left\Vert T\right\Vert _{\mathcal{I}_{BLip}}=\left\Vert u\circ T\right\Vert
	_{\mathcal{I}_{BLip}}$.
\end{theorem}

\begin{proof}
	The \textquotedblleft if" part is follows immediately from the previous
	remark and that $J_{E}:E\longrightarrow \ell _{\infty }(B_{E^{\ast }})$ is a
	metric injection. To prove the \textquotedblleft only if" part, suppose that 
	$\mathcal{I}_{BLip}=\mathcal{I}_{BLip}^{inj}$. Take $T\in
	BLip_{0}(X_{1},X_{2};E)$ and a metric injection $u:E\longrightarrow G$ such
	that such that $u\circ T\in \mathcal{I}_{BLip}(X_{1},X_{2};E)$ and $%
	\left\Vert T\right\Vert _{\mathcal{I}_{BLip}}=\left\Vert u\circ T\right\Vert
	_{\mathcal{I}_{BLip}}$. Since $J_{G}\circ u:E\longrightarrow \ell _{\infty
	}(B_{G^{\ast }})$ is a metric injection, by the metric extension property of 
	$\ell _{\infty }(B_{E^{\ast }})$ there exists $v\in \mathcal{L}(\ell
	_{\infty }(B_{G^{\ast }}),\ell _{\infty }(B_{E^{\ast }}))$ such that $v\circ
	J_{G}\circ u=J_{E}$ and $\left\Vert v\right\Vert =\left\Vert
	J_{E}\right\Vert =1.$ Therefore,%
	\begin{equation*}
	J_{E}\circ T=v\circ J_{G}\circ u\circ T\in \mathcal{I}_{BLip}(X_{1},X_{2};%
	\ell _{\infty }(B_{E^{\ast }})).
	\end{equation*}%
	By assumption, we get $T\in \mathcal{I}_{BLip}(X_{1},X_{2};E)$ and we have 
	\begin{equation*}
	\left\Vert T\right\Vert _{\mathcal{I}_{BLip}}=\left\Vert J_{E}\circ
	T\right\Vert _{\mathcal{I}_{BLip}}=\left\Vert v\circ J_{G}\circ u\circ
	T\right\Vert _{\mathcal{I}_{BLip}}\leq \left\Vert u\circ T\right\Vert _{%
		\mathcal{I}_{BLip}}.
	\end{equation*}%
	The reverse is immediate by the ideal property.
\end{proof}

\begin{example}
By used the previous theorem, we show that the ideal of compact
two-Lipschitz operators is injective. Let $X_{1}$ and $X_{2}$ be pointed
metric spaces and let $E$ be a Banach space. For every $T\in BLip_{0}
\mathcal{(}X_{1},X_{2};E)$ consider $Im_{BLip}(T)$ the bounded subset of $E$
consists all vectors of the form
\begin{equation*}
\frac{T\left( x_{1},x_{2}\right) -T\left( x_{1},x_{2}^{\prime }\right)
	-T\left( x_{1}^{\prime },x_{2}\right) +T\left( x_{1}^{\prime },x_{2}^{\prime
	}\right) }{d(x_{1},x_{1}^{\prime })d(x_{2},x_{2}^{\prime })},
\end{equation*}
where $x_{1},x_{1}^{\prime }\in X_{1},x_{2},x_{2}^{\prime }\in X_{2}$ with $
x_{1}\neq x_{1}^{\prime }$ and $x_{2}\neq x_{2}^{\prime }.$ The mapping $T$
is called compact, in symbols $T\in BLip_{0\mathcal{K}}(X,Y;E)$, if $
Im_{BLip}(T)$ is relatively compact in $E.$ The class $BLip_{0\mathcal{K}}$
is a closed two-Lipschitz operator ideal (see \cite[Section 4.1]{Hamidietall}).
It is clear that if $u:E\longrightarrow G$ is a metric injection, then $
Im_{BLip}(u\circ T)=u\left( Im_{BLip}(T)\right)$. This implies that $BLip_{0
	\mathcal{K}}^{inj}\subset BLip_{0\mathcal{K}}$, thus $BLip_{0\mathcal{K}}$
is injective.	
\end{example}	

\section{Technique to generate two-Lipschitz operators ideals from operator ideals}

Now we introduced a technique to construct a (Banach) two-Lipschitz operator
ideal from a (Banach) linear operator ideal using the the notion of
transpose of a two-Lipschitz operator, was introduced by the second author
in \cite{Dahia} that we now describe briefly. The transpose of a
two-Lipschitz operator $T\in BLip_{0}(X_{1},X_{2};E)$ is the linear operator 
$T^{t}:E^{\ast }\longrightarrow BLip_{0}(X_{1},X_{2};\mathbb{K})$ that maps
each $e^{\ast }\in E^{\ast }$ to $e^{\ast }\circ T$ that is $T^{t}(e^{\ast
})(x,y)=e^{\ast }(T(x,y)),$ for all $(x,y)\in X_{1},X_{2}.$ In addition, we
have $\left\Vert T^{t}\right\Vert =BLip(T)$ (see \cite[Theorem 2.6]{Dahia}).

Recall that the dual of an operator ideal $\mathcal{I}$ is the operator
ideal $\mathcal{I}^{dual}$ such that $u\in \mathcal{I}^{dual}(E,F)$ if and
only if $u^{\ast }\in \mathcal{I}(F^{\ast },E^{\ast })$. Where $u^{\ast
}:F^{\ast }\longrightarrow E^{\ast }$ is the adjoint of $u$ (see \cite[Page
67]{18}). Also the bilinear dual of an operator ideal $\mathcal{I}$ is the
bilinear ideal $\mathcal{I}^{\mathcal{B}\text{-}dual}$ such that for all
Banach spaces $E,F,G$ we have that the bilinear operator $T$ belongs to $%
\mathcal{I}^{\mathcal{B}\text{-}dual}(E,F;G)$ if and only if its adjoint $%
T^{\ast }$ belongs to $\mathcal{I}(G^{\ast },\mathcal{L}(E,F;\mathbb{K})$
(see \cite[Definition 1.2]{Botelho14} for the polynomial case). Note that if 
$\mathcal{I}$ is a normed operator ideal we define $\left\Vert u\right\Vert
_{\mathcal{I}^{dual}}:=\left\Vert u^{\ast }\right\Vert _{\mathcal{I}}$ and $%
\left\Vert T\right\Vert _{\mathcal{I}^{\mathcal{B}\text{-}dual}}:=\left\Vert
T^{\ast }\right\Vert _{\mathcal{I}}$.

The proof of the following lemma is an adaptation of \cite[Theorem 2.2]%
{Botelho14}.

\begin{lemma}
	Let $\mathcal{I}$ be a normed operator ideal. Then $\mathcal{I}^{\mathcal{B}%
		\text{-}dual}=\mathcal{I}^{dual}\circ \mathcal{B}$ and $\left\Vert \cdot
	\right\Vert _{\mathcal{I}^{\mathcal{B}\text{-}dual}}=\left\Vert \cdot
	\right\Vert _{\mathcal{I}^{dual}\circ \mathcal{B}}$. \newline
	Where $\mathcal{B}$ is the ideal of bilinear operators.
\end{lemma}

Now we arrive to define and characterize the two-Lipschitz dual of an
operator ideal.

\begin{definition}
	The two-Lipschitz dual of a given operator ideal $\mathcal{I}$ is defined by 
	\begin{eqnarray*}
		&&\mathcal{I}^{BLip\text{-}dual}(X_{1},X_{2};E) \\
		&=&\left\{ T\in BLip_{0}(X_{1},X_{2};E):T^{t}\in \mathcal{I}(E^{\ast
		},BLip_{0}(X_{1},X_{2};\mathbb{K})\right\} .
	\end{eqnarray*}
\end{definition}

If $(\mathcal{I},\left\Vert \cdot \right\Vert _{\mathcal{I}})$ is a normed
operator ideal, define $\left\Vert T\right\Vert _{\mathcal{I}^{BLip\text{-}%
		dual}}:=\left\Vert T^{t}\right\Vert _{\mathcal{I}}$ for every $T\in \mathcal{%
	I}^{BLip\text{-}dual}(X_{1},X_{2};E).$

\begin{theorem} \label{dualth}
	We have $\mathcal{I}^{BLip\text{-}dual}=\mathcal{I}^{dual}\circ BLip_{0}.
	$ Moreover, if $(\mathcal{I},\left\Vert \cdot \right\Vert _{\mathcal{I}})$
	is a normed operator ideal, then $\left\Vert \cdot \right\Vert _{\mathcal{I}%
		^{BLip\text{-}dual}}=\left\Vert \cdot \right\Vert _{\mathcal{I}
		^{dual}\circ BLip_{0}}.$
\end{theorem}

\begin{proof}
	Let $X_{1},X_{2}$ be pointed metric spaces and $E$ be Banach space. Suppose
	that $T\in \mathcal{I}^{BLip_{0}\text{-}dual}(X_{1},X_{2};E)$, that is $
	T^{t}\in \mathcal{I}(E^{\ast },BLip_{0}(X_{1},X_{2};\mathbb{K})$. Using \cite
	[Lemma 2.9]{Dahia} and the ideal property we get 
	\[
	(T_{B})^{\ast }=R\circ T^{t}\in \mathcal{I}(E^{\ast },\mathcal{B}(\text{\AE }
	\left( X_{1}\right) ,\text{\AE }\left( X_{2}\right) ;\mathbb{K}),
	\]%
	where $R:BLip_{0}(X_{1},X_{2};\mathbb{K})\longrightarrow \mathcal{B}($\AE $
	\left( X_{1}\right) ,$\AE $\left( X_{2}\right) ;\mathbb{K})$ is the
	isomorphic isometry given by $R(\phi )=\phi _{B}$. Then
	\[
	T_{B}\in \mathcal{I}^{\mathcal{B}\text{-}dual}(\text{\AE }\left(
	X_{1}\right) ,\text{\AE }\left( X_{2}\right) ;E).
	\]
	By the previous lemma, it follows that $T_{B}\in \mathcal{I}^{dual}\circ 
	\mathcal{B}($\AE $\left( X_{1}\right) ,$\AE $\left( X_{2}\right) ;E).$ By 
	\cite[Proposition 3.2]{BPR07} we have that $T_{L}\in \mathcal{I}^{dual}($\AE 
	$\left( X_{1}\right) \widehat{\otimes }_{\pi }$\AE $\left( X_{2}\right) ,E)$
	. Since $T=T_{L}\circ \sigma _{2}\circ (\delta _{X_{1}},\delta _{X_{2}})$,
	where $\sigma _{2}$ is the canonical bilinear operator defined from \AE $
	\left( X_{1}\right) \times $\AE $\left( X_{2}\right) $ to \AE $\left(
	X_{1}\right) \widehat{\otimes }_{\pi }$\AE $\left( X_{2}\right) $ by $\sigma
	_{2}(m_{x0},m_{y0})=m_{x0}\otimes m_{y0},$ we conclude $T\in \mathcal{I}
	^{dual}\circ BLip_{0}(X_{1},X_{2};E)$. Besides, it follows from (\ref{compo}
	) and \cite[Proposition 3.7]{BPR07} that
	\begin{eqnarray*}
		\left\Vert T\right\Vert _{\mathcal{I}^{dual}\circ BLip_{0}} &=&\left\Vert
		T_{L}\right\Vert _{\mathcal{I}^{dual}}=\left\Vert T_{B}\right\Vert _{
			\mathcal{I}^{dual}\circ \mathcal{B}} \\
		&=&\left\Vert T_{B}\right\Vert _{\mathcal{I}^{\mathcal{B}\text{-}
				dual}}=\left\Vert R\circ T^{t}\right\Vert _{\mathcal{I}} \\
		&\leq &\left\Vert R\right\Vert \left\Vert T^{t}\right\Vert _{\mathcal{I}
		}=\left\Vert T\right\Vert _{\mathcal{I}^{BLip\text{-}dual}}.
	\end{eqnarray*}
	Conversely, assume that $T\in \mathcal{I}^{dual}\circ BLip_{0}(X_{1},X_{2};E)
	$. For all $\varepsilon >0$ choose Banach space $F,$ a linear operator $
	u:F\longrightarrow E$ and $S\in BLip_{0}(X_{1},X_{2};F)$ such that $T=u\circ
	S$ and $u^{\ast }\in \mathcal{I(}E^{\ast },F^{\ast })$ such that
	\[
	\left\Vert u\right\Vert _{\mathcal{I}^{dual}}BLip(S)\leq \varepsilon
	+\left\Vert T\right\Vert _{\mathcal{I}^{dual}\circ BLip_{0}}.
	\]
	By the ideal property we have $T^{t}=S^{t}\circ u^{\ast }\in \mathcal{I}
	(E^{\ast },BLip_{0}(X_{1},X_{2};\mathbb{K}),$ that is $T\in \mathcal{I}
	^{BLip_{0}\text{-}dual}(X_{1},X_{2};E)$. Moreover,
	\begin{eqnarray*}
		\left\Vert T\right\Vert _{\mathcal{I}^{BLip_{0}\text{-}dual}} &=&\left\Vert
		S^{t}\circ u^{\ast }\right\Vert _{\mathcal{I}}\leq \left\Vert
		S^{t}\right\Vert \left\Vert u^{\ast }\right\Vert _{\mathcal{I}} \\
		&=&BLip(S)\left\Vert u\right\Vert _{\mathcal{I}^{dual}} \\
		&\leq &\varepsilon +\left\Vert T\right\Vert _{\mathcal{I}^{dual}\circ
			BLip_{0}},
	\end{eqnarray*}
	which completes the proof.
\end{proof}
By employing composition method, \cite[Proposition 3.6]{Hamidietall}, and
applying the formula presented in the earlier theorem, we obtain the
following result.

\begin{corollary}
	$\mathcal{I}^{BLip\text{-}dual}$ is a Banach two-Lipschitz operator
	ideal whenever $\mathcal{I}$ a Banach operator ideal.
\end{corollary}

Schauder type theorem on the compactness of the transpose of two-Lipschitz
operator (\cite[Theorem 2.10]{Dahia}) can be demonstrate directly based on
our previous results. Let $\mathcal{K}$ be the ideal of linear compact
operators. It is well-known $\mathcal{K}^{dual}=\mathcal{K}$ (see \cite[Page
68]{18}).

\begin{corollary}
	A two-Lipschitz operator $T:X_{1}\times X_{2}\longrightarrow E$ is compact
	if and only if its transpose $T^{t}$ is a compact linear operator.
\end{corollary}

\begin{proof}
	Through the utilization of \cite[Proposition 3.6]{Hamidietall} and Theorem 
	\ref{dualth} we get
	\begin{equation}
	BLip_{0\mathcal{K}}=\mathcal{K}\circ BLip_{0}=\mathcal{K}^{dual}\circ
	BLip_{0}=\mathcal{K}^{BLip\text{-}dual}.  \label{lastfr}
	\end{equation}
	The proof conclude.
\end{proof}

\begin{example}
	(a) The formula (\ref{lastfr}) assert that the ideal of two-Lipschtz compact
	operators is generated by the duality method from the operator ideal $
	\mathcal{K}.$ \newline
(b) For $1<p<\infty $, consider $\Pi _{p}$ the Banach ideal of $p$-summing
linear operators was introduced by Pietsch in \cite{17}. Cohen in \cite
{Coh73} presented the Banach ideal $\mathcal{D}_{p}$ of strongly $p$-summing
linear operators and prove that $\mathcal{D}_{q}=(\Pi _{p})^{dual}$, where $
1/p+1/q=1$, (see \cite[Theorem 2.2.2]{Coh73}). Let $\mathcal{D}_{q}^{BL}$ be
the Banach ideal of strongly two-Lipschitz $q$-summing operators generated
by the composition method from $\mathcal{D}_{q}$, that is $\mathcal{D}
_{q}^{BL}=\mathcal{D}_{q}\circ BLip_{0}$ (see \cite[Corollary 4.13]
{Hamidietall}). By Theorem \ref{dualth} we obtain $(\Pi _{p})^{BLip\text{-}dual}\mathcal{=}\mathcal{D}_{q}^{BL}$.
\end{example}	
\section{Techniques to generate two-Lipschitz operators ideals from Lipschitz operator ideals}
There are different ways of constructing an ideal of multilinear mappings
from a given operator ideal (see \cite{99}, \cite{Brau84}). In this section
we fill this gap by developing two techniques to generate two-Lipschitz
operator ideals from given Lipschitz operator ideals.

\subsection{Lipschitzization method}

Let $X_{1}$ and $X_{2}$ be pointed metric spaces and let $E$ be Banach spac.
Each two-Lipschitz operator $T:X_{1}\times X_{2}\longrightarrow E$ is
associated with the mappings $T_{1}:x_{1}\longmapsto T_{1}(x_{1})$ and $%
T_{2}:x_{2}\longmapsto T_{2}(x_{2})$ for all $x_{1}\in X_{1}$ and $x_{2}\in
X_{2}$. Where $T_{1}(x_{1}):X_{2}\longrightarrow E,$ $T_{2}(x_{2}):X_{2}%
\longrightarrow E$ are the mappings defined by 
\begin{equation*}
T_{1}(x_{1})(x_{2})=T(x_{1},x_{2})=T_{2}(x_{2})(x_{1}).
\end{equation*}%
Thanks to an argument detailed in \cite[Proposition 2.2]{Hamidietall}, we
prove that $T_{1}\in Lip_{0}\left( X_{1},Lip_{0}(X_{2},E)\right) $ and $
T_{2}\in Lip_{0}\left( X_{2},Lip_{0}(X_{1},E)\right) $. Furthermore
\begin{equation}
Lip(T_{1})=BLip(T)=Lip(T_{2}).  \label{nrequa}
\end{equation}

\begin{definition}
	Let $\mathcal{I}_{Lip}^{1},\mathcal{I}_{Lip}^{2}$ be Lipschitz operator
	ideals, a two-Lipschitz operator $T\in BLip(X_{1},X_{2};E)$ is said to be of
	Lipschitzization type, in symbols $T\in \lbrack \mathcal{I}_{Lip}^{1},%
	\mathcal{I}_{Lip}^{2}](X_{1},X_{2};E)$, if $T_{1}\in \mathcal{I}%
	_{Lip}^{1}\left( X_{1},Lip_{0}(X_{2},E)\right) $ and $T_{2}\in \mathcal{I}%
	_{Lip}^{2}\left( X_{2},Lip_{0}(X_{1},E)\right) $.
\end{definition}

If $\mathcal{I}_{Lip}^{1},\mathcal{I}_{Lip}^{2}$ are normed Lipschitz
operator ideals and $T\in \lbrack \mathcal{I}_{Lip}^{1},\mathcal{I}%
_{Lip}^{2}](X_{1},X_{2};E)$ we define 
\begin{equation}
\left\Vert T\right\Vert _{[\mathcal{I}_{Lip}^{1},\mathcal{I}%
	_{Lip}^{2}]}=\max \left\{ \left\Vert T_{1}\right\Vert _{\mathcal{I}%
	_{Lip}^{1}},\left\Vert T_{2}\right\Vert _{\mathcal{I}_{Lip}^{2}}\right\} .
\label{norm}
\end{equation}

\begin{remark}
	\label{rem}
	
	\noindent (a) Let $T\in \lbrack \mathcal{I}_{Lip}^{1},\mathcal{I}%
	_{Lip}^{2}](X_{1},X_{2};E)$, then $Lip(T_{j})\leq \left\Vert
	T_{j}\right\Vert _{\mathcal{I}_{Lip}^{j}}$, $(j=1,2).$ By (\ref{nrequa}) and
	(\ref{norm}) we obtain 
	\begin{equation}
	BLip(T)\leq \left\Vert T\right\Vert _{[\mathcal{I}_{Lip}^{1},\mathcal{I}%
		_{Lip}^{2}]}.  \label{rm}
	\end{equation}
	
	\noindent (b) It is clear that if $f\in X_{1}^{\#}$, $g\in X_{2}^{\#}$ and $%
	e\in E,$ then the mapping $f\cdot g\cdot e$ defined in Example \ref{remex}
	is belongs to $[\mathcal{I}_{Lip}^{1},\mathcal{I}_{Lip}^{2}](X_{1},X_{2};E).$
	
	\noindent (c) Let $id_{\mathbb{K}^{2}}$ be the two-Lipschitz operator
	defined in Example \ref{remex}. It is easy to see that 
	\begin{equation*}
	(id_{\mathbb{K}^{2}})_{1}=u\circ id_{\mathbb{K}}\in \mathcal{I}_{Lip}^{1}(%
	\mathbb{K},Lip_{0}(\mathbb{K},\mathbb{K})),
	\end{equation*}%
	where $u$ is the linear operator $u:\mathbb{K}\longrightarrow Lip_{0}(%
	\mathbb{K},\mathbb{K})$ defined by $u(\alpha )=\alpha id_{\mathbb{K}}$ with $%
	\left\Vert u\right\Vert =1$. Furthermore, $\left\Vert (id_{\mathbb{K}%
		^{2}})_{1}\right\Vert _{\mathcal{I}_{Lip}^{1}}=1$. In a similar way one can
	prove that $\left\Vert (id_{\mathbb{K}^{2}})_{2}\right\Vert _{\mathcal{I}%
		_{Lip}^{2}}=1$ and then $\left\Vert id_{\mathbb{K}^{2}}\right\Vert _{[%
		\mathcal{I}_{Lip}^{1},\mathcal{I}_{Lip}^{2}]}=1.$
\end{remark}

The above remark and the next proposition asserts that $[\mathcal{I}%
_{Lip}^{1},\mathcal{I}_{Lip}^{2}]$ is a normed two-Lipschitz operator ideal
whenever $\mathcal{I}_{Lip}^{1}$ and $\mathcal{I}_{Lip}^{2}$ are normed
Lipschitz operator ideals. This method of constructing an ideal of
two-Lipschitz operators from ideals of Lipschitz operators is called the 
\emph{Lipschitzization method}.

\begin{proposition}
	Let $\mathcal{I}_{Lip}^{1},\mathcal{I}_{Lip}^{2}$ be normed Lipschitz
	operator ideals, $X_{1},X_{2}$ be pointed metric spaces and $E$ be Banach
	space.
	
	\begin{enumerate}
		\item $[\mathcal{I}_{Lip}^{1},\mathcal{I}_{Lip}^{2}](X_{1},X_{2};E)$ is a
		linear subspace of $BLip_{0}(X_{1},X_{2};E)$.
		
		\item $\left\Vert \cdot \right\Vert _{[\mathcal{I}_{Lip}^{1},\mathcal{I}%
			_{Lip}^{2}]}$ is a norm on $[\mathcal{I}_{Lip}^{1},\mathcal{I}%
		_{Lip}^{2}](X_{1},X_{2};E)$.
		
		\item (Ideal property). Let $f\in Lip_{0}(Z,X_{1})$, $g\in Lip_{0}(W,X_{2})$
		and $u\in \mathcal{L}(E,F)$. If $T\in \lbrack \mathcal{I}_{Lip}^{1},\mathcal{%
			I}_{Lip}^{2}](X_{1},X_{2};E)$, then the composition $u\circ T\circ (f,g)$ is
		in $[\mathcal{I}_{Lip}^{1},\mathcal{I}_{Lip}^{2}](Z,W;F)$ and 
		\begin{equation*}
		\Vert u\circ T\circ (f,g)\Vert _{\lbrack \mathcal{I}_{Lip}^{1},\mathcal{I}%
			_{Lip}^{2}]}\leq \left\Vert u\right\Vert \Vert T\Vert _{\lbrack \mathcal{I}%
			_{Lip}^{1},\mathcal{I}_{Lip}^{2}]}Lip(f)Lip(g).
		\end{equation*}
	\end{enumerate}
\end{proposition}

\begin{proof}
	(1) Let $T,T^{\prime }\in \lbrack \mathcal{I}_{Lip}^{1},\mathcal{I}%
	_{Lip}^{2}](X_{1},X_{2};E)$ and $\alpha \in \mathbb{K}$. We have that 
	\begin{equation*}
	(\alpha T+T^{\prime })_{j}=\alpha T_{j}+T_{j}^{\prime }\in \mathcal{I}%
	_{Lip}^{j}.
	\end{equation*}
	
	(2) Is straightforward.
	
	(3) Let $f\in Lip_{0}(Z,X_{1})$, $g\in Lip_{0}(W,X_{2}),$ $u\in \mathcal{L}%
	(E,F)$ and $T\in \lbrack \mathcal{I}_{Lip}^{1},\mathcal{I}%
	_{Lip}^{2}](X_{1},X_{2};E).$ Consider the linear mapping $\psi
	:Lip_{0}(X_{2},E)\longrightarrow Lip_{0}(W,F)$ given by $\psi (R)=u\circ
	R\circ g,$ where $R\in Lip_{0}(X_{2},E).$ A simple calculation shows that 
	\begin{equation*}
	\left( u\circ T\circ (f,g)\right) _{1}=\psi \circ T_{1}\circ f\in \mathcal{I}%
	_{Lip}^{1}(Z;Lip_{0}(W,F)).
	\end{equation*}%
	The mapping $\psi $ is bounded and $\left\Vert \psi \right\Vert \leq
	\left\Vert u\right\Vert Lip(g),$ in order to see this we have 
	\begin{eqnarray*}
		\left\Vert \psi (R)\right\Vert  &=&\underset{w,w^{\prime }\in W,w\neq
			w^{\prime }}{\sup }\frac{\left\Vert \psi (R)(w)-\psi (R)(w^{\prime
			})\right\Vert }{d(w,w^{\prime })} \\
		&\leq &\left\Vert u\right\Vert Lip(g)Lip(R).
	\end{eqnarray*}%
	It follows that 
	\begin{eqnarray*}
		\left\Vert \left( u\circ T\circ (f,g)\right) _{1}\right\Vert _{\mathcal{I}%
			_{Lip}^{1}} &\leq &\left\Vert \psi \right\Vert \left\Vert T_{1}\right\Vert _{%
			\mathcal{I}_{Lip}^{1}}Lip(f) \\
		&\leq &\left\Vert u\right\Vert \left\Vert T_{1}\right\Vert _{\mathcal{I}%
			_{Lip}^{1}}Lip(f)Lip(g).
	\end{eqnarray*}%
	By a similar argument, we obtain 
	\begin{equation*}
	\left\Vert \left( u\circ T\circ (f,g)\right) _{2}\right\Vert _{\mathcal{I}%
		_{Lip}^{2}}\leq \left\Vert u\right\Vert \left\Vert T_{2}\right\Vert _{%
		\mathcal{I}_{Lip}^{2}}Lip(f)Lip(g),
	\end{equation*}%
	hence 
	\begin{equation*}
	\Vert u\circ T\circ (f,g)\Vert _{\left[ \mathcal{I}_{Lip}^{1},\mathcal{I}%
		_{Lip}^{2}\right] }\leq \left\Vert u\right\Vert \Vert T\Vert _{\lbrack 
		\mathcal{I}_{Lip}^{1},\mathcal{I}_{Lip}^{2}]}Lip(f)Lip(g).
	\end{equation*}
\end{proof}
It is necessary to investigate the closedness of the two-Lipschitz operator
ideal $[\mathcal{I}_{Lip}^{1},\mathcal{I}_{Lip}^{2}]$.

\begin{proposition}
	Let $\mathcal{I}_{Lip1}$ and $\mathcal{I}_{Lip2}$ be Lipschitz operator
	ideals. Then, 
	\begin{equation*}
	\overline{\lbrack \mathcal{I}_{Lip}^{1},\mathcal{I}_{Lip}^{2}]}\subset
	\lbrack \overline{\mathcal{I}_{Lip}^{1}},\overline{\mathcal{I}_{Lip}^{2}}].
	\end{equation*}%
	Consequently, if $\mathcal{I}_{Lip}^{1}$ and $\mathcal{I}_{Lip}^{2}$ are
	closed, then $[\mathcal{I}_{Lip}^{1},\mathcal{I}_{Lip}^{2}]$ is closed too.
\end{proposition}

\begin{proof}
	Let $T\in \overline{[\mathcal{I}_{Lip}^{1},\mathcal{I}_{Lip}^{2}]}%
	(X_{1},X_{2};E)$. For each $\varepsilon >0$ we can choose two-Lipschitz
	operators $A,B:X_{1}\times X_{2}\longrightarrow E$ such that $A\in \lbrack 
	\mathcal{I}_{Lip}^{1},\mathcal{I}_{Lip}^{2}](X_{1},X_{2};E)$ with $%
	BLip(B)<\varepsilon $ and $T=A+B$. It is easy to see that $T_{1}=A_{1}+B_{1}.
	$ We have $A_{1}\in \mathcal{I}_{Lip}^{1}(X_{1},Lip_{0}(X_{2},E))$ and by (%
	\ref{nrequa}) we get $Lip(B_{1})<\varepsilon $, this means that $T_{1}\in 
	\overline{\mathcal{I}_{Lip}^{1}}(X_{1},Lip_{0}(X_{2},E))$. By a similar
	argument we proof that $T_{2}\in \overline{\mathcal{I}_{Lip}^{2}}%
	(X_{2},Lip_{0}(X_{1},E)).$ Hence, $T\in \lbrack \overline{\mathcal{I}%
		_{Lip}^{1}},\overline{\mathcal{I}_{Lip}^{2}}](X_{1},X_{2};E).$
\end{proof}
Following the idea of \cite[Page 309]{uniform}, our next aim is to show the
injectivity of the two-Lipschitz operator ideal constructed by the
Lipschitzization method. For the proof, we need the following inclusion results.

\begin{proposition}
	Let $\mathcal{I}_{Lip}^{1},\mathcal{I}_{Lip}^{2}$ be normed Lipschitz
	operator ideals. Then, 
	\begin{equation*}
	\lbrack \mathcal{I}_{Lip}^{1},\mathcal{I}_{Lip}^{2}]^{inj}\subset \lbrack (%
	\mathcal{I}_{Lip}^{1})^{inj},(\mathcal{I}_{Lip}^{2})^{inj}].
	\end{equation*}%
	Moreover we have $\left\Vert \cdot \right\Vert _{[(\mathcal{I}%
		_{Lip}^{1})^{inj},(\mathcal{I}_{Lip}^{2})^{inj}]}\leq \left\Vert \cdot
	\right\Vert _{[\mathcal{I}_{Lip}^{1},\mathcal{I}_{Lip}^{2}]^{inj}}.$
\end{proposition}

\begin{proof}
	Take $T\in \lbrack \mathcal{I}_{Lip}^{1},\mathcal{I}_{Lip}^{2}]^{inj}$ and
	consider the metric injection $j_{1}:Lip_{0}(X_{2},E)\longrightarrow
	Lip_{0}(X_{2},\ell _{\infty }(B_{E^{\ast }}))$ given by $j_{1}(f):=J_{E}%
	\circ f.$ For any $x_{1}\in X_{1}$ and $x_{2}\in X_{2}$ we have 
	\begin{eqnarray*}
		j_{1}\left( T_{1}(x_{1})\right) (x_{2}) &=&\left( J_{E}\circ
		T_{1}(x_{1})\right) (x_{2}) \\
		&=&J_{E}\left( T(x_{1},x_{2})\right)  \\
		&=&\left( J_{E}\circ T\right) _{1}(x_{1})(x_{2}),
	\end{eqnarray*}%
	showing that $j_{1}\circ T_{1}=\left( J_{E}\circ T\right) _{1}.$ By the
	metric extension property of the space $\ell _{\infty
	}(B_{Lip_{0}(X_{2},E)^{\ast }})$, we find a linear operator $\phi _{1}$ from 
	$Lip_{0}(X_{2},\ell _{\infty }(B_{E^{\ast }}))$ to $\ell _{\infty
	}(B_{Lip_{0}(X_{2},E)^{\ast }})$ such that $\phi _{1}\circ
	j_{1}=J_{Lip_{0}(X_{2},E)}$ and $\left\Vert \phi _{1}\right\Vert =1.$ By
	this notations we have 
	\begin{equation*}
	J_{Lip_{0}(X_{2},E)}\circ T_{1}=\phi _{1}\circ \left( J_{E}\circ T\right)
	_{1}\in \mathcal{I}_{Lip1}(X_{1},\ell _{\infty }\left(
	B_{Lip_{0}(X_{2},E)^{\ast }})\right) ,
	\end{equation*}%
	It results that $T_{1}$ belongs to $(\mathcal{I}_{Lip}^{1})^{inj}\left(
	X_{1},Lip_{0}(X_{2},E\right) .$ In addition, 
	\begin{eqnarray*}
		\left\Vert T_{1}\right\Vert _{(\mathcal{I}_{Lip}^{1})^{inj}} &\leq
		&\left\Vert \phi _{1}\right\Vert \left\Vert \left( J_{E}\circ T\right)
		_{1}\right\Vert _{\mathcal{I}_{Lip}^{1}} \\
		&\leq &\left\Vert J_{E}\circ T\right\Vert _{[\mathcal{I}_{Lip}^{1},\mathcal{I%
			}_{Lip}^{2}]} \\
		&=&\left\Vert T\right\Vert _{[\mathcal{I}_{Lip}^{1},\mathcal{I}%
			_{Lip}^{2}]^{inj}}.
	\end{eqnarray*}%
	By a similar argument, we get $T_{2}\in (\mathcal{I}_{Lip}^{2})^{inj}\left(
	X_{2},Lip_{0}(X_{1},E\right) $ and 
	\begin{equation*}
	\left\Vert T_{2}\right\Vert _{(\mathcal{I}_{Lip}^{2})^{inj}}\leq \left\Vert
	T\right\Vert _{[\mathcal{I}_{Lip}^{1},\mathcal{I}_{Lip}^{2}]^{inj}}.
	\end{equation*}%
	Thus, $T$ belongs to $[(\mathcal{I}_{Lip}^{1})^{inj},(\mathcal{I}%
	_{Lip}^{2})^{inj}]$ with $\left\Vert T\right\Vert _{[(\mathcal{I}%
		_{Lip}^{1})^{inj},(\mathcal{I}_{Lip}^{2})^{inj}]}\leq \left\Vert
	T\right\Vert _{[\mathcal{I}_{Lip}^{1},\mathcal{I}_{Lip}^{2}]^{inj}}.$
\end{proof}
\begin{corollary}
	If $\mathcal{I}_{Lip1}$ and $\mathcal{I}_{Lip2}$ are injective Lipschitz
	operator ideals, then $\left[ \mathcal{I}_{Lip1},\mathcal{I}_{Lip2}\right] $
	is injective two-Lipschitz operator ideal.
\end{corollary}

\begin{proof}
	It it follows from the above proposition that
	\begin{equation*}
	\left[ \mathcal{I}_{Lip}^{1},\mathcal{I}_{Lip}^{2}\right] ^{inj}\subset %
	\left[ (\mathcal{I}_{Lip}^{1})^{inj},(\mathcal{I}_{Lip}^{2})^{inj}\right] =%
	\left[ \mathcal{I}_{Lip}^{1},\mathcal{I}_{Lip}^{2}\right] .
	\end{equation*}
\end{proof}

Similar to the multilinear caes (see \cite{Floret}), we define the notion of
symmetric two-Lipschitz operator and we investigate the symmetry of the
Lipschitzization ideals.

In what follows we take $X_{2}=X_{2}=X.$ Note that if $T\in BLip_{0}(X,X;E)$%
, the symmetric two-Lipschitz operator $T^{s}\in BLip_{0}(X,X;E)$ associated
with $T$ is given by%
\begin{equation*}
T^{s}(x_{1},x_{2})=\frac{1}{2}\left( T(x_{1},x_{2})+T(x_{2},x_{1})\right) ,
\end{equation*}%
for every $x_{1},x_{2}\in X.$ A simple calculation gives 
\begin{equation}
(T^{s})_{j}=\frac{1}{2}\left( T_{1}+T_{2}\right) ,\text{ }j=1,2.  \label{sm}
\end{equation}

\begin{definition}
	The two-Lipschitz operator ideal $\mathcal{I}_{BLip}$ is said to be
	symmetric if $T^{s}\in \mathcal{I}_{BLip}(X,X;E)$ whenever $T\in \mathcal{I}%
	_{BLip}(X,X;E).$
\end{definition}

\begin{theorem}
	Let $\mathcal{I}_{Lip}^{1}$ and $\mathcal{I}_{Lip}^{2}$ be Lipschitz
	operator ideals. Then the following assertions are equivalent.
	
	\begin{enumerate}
		\item[(1)] The two-Lipschitz operator ideal $\left[ \mathcal{I}_{Lip}^{1},%
		\mathcal{I}_{Lip}^{2}\right] $ is symmetric.
		
		\item[(2)] $\left[ \mathcal{I}_{Lip}^{1},\mathcal{I}_{Lip}^{2}\right] =\left[
		\mathcal{I}_{Lip}^{2},\mathcal{I}_{Lip}^{1}\right] .$
		
		\item[(3)] $\mathcal{I}_{Lip}^{1}=\mathcal{I}_{Lip}^{2}.$
	\end{enumerate}
\end{theorem}

\begin{proof}
	$(2)\Longrightarrow (1)$ Let $T\in \left[ \mathcal{I}_{Lip}^{1},\mathcal{I}%
	_{Lip}^{2}\right] (X,X;E)$. By the hypothesis we have $T_{j}\in \mathcal{I}%
	_{Lip}^{k}(X,Lip_{0}(X,E))$ and by using (\ref{sm}) we get $(T^{s})_{j}\in 
	\mathcal{I}_{Lip}^{k}(X,Lip_{0}(X,E))$ for all $j,k=1,2.$ Then $T^{s}\in %
	\left[ \mathcal{I}_{Lip}^{1},\mathcal{I}_{Lip}^{2}\right] (X,X;E)$.
	
	$(1)\Longrightarrow (3)$ Let $\phi \in \mathcal{I}_{\mathcal{I}%
		_{Lip}^{1}}(X,E)$. Fix $f\in X^{\#}$ and $a\in X$ with $f(a)=1$. Consider a
	two-Lipschitz operator $T\in BLip_{0}(X,X;E)$ and a linear operator $u\in 
	\mathcal{L(}E,Lip_{0}(X,E))$ given by $T(x_{1},x_{2}):=\phi (x_{1})f(x_{2})$
	and $u(e)(x_{2}):=f(x_{2})e$, for all $x_{1},x_{2}\in X$ and $e\in E$. It
	follow that $u\circ \phi =T_{1}\in \mathcal{I}_{\mathcal{I}%
		_{Lip}^{1}}(X,Lip_{0}(X,E))$. On the other hand, for every $x_{2}\in X$ we
	have $T_{2}(x_{2})=f(x_{2})\phi $, then $T_{2}\in \mathcal{I}_{\mathcal{I}%
		_{Lip}^{2}}(X,Lip_{0}(X,E))$. By the hypothesis we have that $T^{s}\in \left[
	\mathcal{I}_{Lip}^{1},\mathcal{I}_{Lip}^{2}\right] (X,X;E)$, from which it
	follows that $T_{1}=2(T^{s})_{2}-T_{2}\in \mathcal{I}_{\mathcal{I}%
		_{Lip}^{2}}(X,Lip_{0}(X,E))$. If we define the linear operator $%
	v:Lip_{0}(X,E)\longrightarrow E$ by $v(f):=f(a)$, we obtain $v\circ
	T_{1}=\phi \in \mathcal{I}_{\mathcal{I}_{Lip}^{2}}(X,E)$. Therefore $%
	\mathcal{I}_{\mathcal{I}_{Lip}^{1}}\subset \mathcal{I}_{\mathcal{I}%
		_{Lip}^{2}}$. Using a similar argument as above we get the reverse inclusion.
	
	The direction $(3)\Longrightarrow (2)$ is obvious.
\end{proof}

\subsection{Factorization method}
Given the Lipschitz operator ideals $\mathcal{I}_{Lip}^{1},\mathcal{I}%
_{Lip}^{2}$, a two-Lipschitz operator $T\in BLip_{0}(X_{1},X_{2};E)$ is said
to be of type $BLip(\mathcal{I}_{Lip}^{1},\mathcal{I}_{Lip}^{2})$, in
symbols $T\in BLip(\mathcal{I}_{Lip}^{1},\mathcal{I}%
_{Lip}^{2})(X_{1},X_{2};E)$, if there are Banach spaces $G_{1},G_{2}$,
Lipschitz operators $f_{j}\in \mathcal{I}_{Lip}^{j}(X_{j},G_{j})$ , $j=1,2$,
and a two-Lipschitz operator $R\in BLip_{0}(G_{1},G_{2};E)$ such that $%
T=R\circ (f_{1},f_{2})$. If $\mathcal{I}_{Lip}^{1},\mathcal{I}_{Lip}^{2}$
are normed Lipschitz operator ideals and $T\in BLip(\mathcal{I}_{Lip}^{1},%
\mathcal{I}_{Lip}^{2})(X_{1},X_{2};E)$ we define 
\begin{equation*}
\left\Vert T\right\Vert _{BLip(\mathcal{I}_{Lip}^{1},\mathcal{I}%
	_{Lip}^{2})}=\inf BLip(R)\left\Vert f_{1}\right\Vert _{\mathcal{I}%
	_{Lip}^{1}}\left\Vert f_{2}\right\Vert _{\mathcal{I}_{Lip}^{2}},
\end{equation*}%
where the infimum is taken over all possible factorizations $T=R\circ
(f_{1},f_{2})$ with $f_{j}$ belonging to $\mathcal{I}_{Lip}^{j},(j=1,2)$ and
the two-Lipschitz operator $R$.

\begin{remark}
	\label{rpop1} If $T\in BLip(\mathcal{I}_{Lip}^{1},\mathcal{I}%
	_{Lip}^{2})(X_{1},X_{2};E)$ then 
	\begin{equation}
	BLip(T)\leq \left\Vert T\right\Vert _{BLip(\mathcal{I}_{Lip}^{1},\mathcal{I}%
		_{Lip}^{2})}.  \label{pop2}
	\end{equation}%
	Indeed, there are Banach spaces $G_{1},G_{2}$, $f_{j}\in \mathcal{I}%
	_{Lip}^{j}(X_{j},G_{j})$, $j=1,2$, and $R\in BLip_{0}(G_{1},G_{2};E)$ such
	that $T=R\circ (f_{1},f_{2})$. By \cite[Proposition 2.3]{Hamidietall} we
	have 
	\begin{equation*}
	BLip(T)\leq BLip(R)Lip(f_{1})Lip(f_{2})\leq BLip(R)\left\Vert
	f_{1}\right\Vert _{\mathcal{I}_{Lip}^{1}}\left\Vert f_{2}\right\Vert _{%
		\mathcal{I}_{Lip}^{2}}.
	\end{equation*}%
	Passing to the infimum over all $R,f_{1}$ and $f_{2}$ we get (\ref{pop2}).
\end{remark}

\begin{example}
	\label{remex} Let us show some basic ways of constructing two-Lipschitz
	operator belonging to $BLip(\mathcal{I}_{Lip}^{1},\mathcal{I}_{Lip}^{2})$.
	Let $X_{1}$ and $X_{2}$ be pointed metric spaces and let $E$ be Banach spac.
	
	\noindent (a) Consider non-zero Lipschitz functions $f\in X_{1}^{\#}$, $g\in
	X_{2}^{\#}$ and $e\in E.$ The mapping $f\cdot g\cdot e:X_{1}\times
	X_{2}\longrightarrow E$ defined by $f\cdot g\cdot e(x,y)=f(x)g(y)e$ is a
	two-Lipschitz operator with $BLip(f\cdot g\cdot e)=Lip(f)Lip(g)\left\Vert
	e\right\Vert $ (see \cite[Page 7]{Hamidietall}). \textit{Every }mappings of
	this form is belongs\textit{\ to }$BLip(\mathcal{I}_{Lip}^{1},\mathcal{I}%
	_{Lip}^{2})$ .\ A simple calculation shows this result by taking into
	account that $f\cdot g\cdot e=B\circ (f,g)$, where $B:\mathbb{K}%
	^{2}\longrightarrow E$ is the two-Lipschitz operator defined by $B(\alpha
	,\beta )=\alpha \beta e$.
	
	\noindent (b) \textit{Let us show a particular example of the case mentioned
		above. Let }$id_{\mathbb{K}^{2}}$ be the two-Lipschitz operator $id_{\mathbb{%
			K}^{2}}:\mathbb{K}^{2}\mathbb{\longrightarrow K}$ \textit{given by} $id_{%
		\mathbb{K}^{2}}(\alpha ,\beta )=\alpha \beta .$\textit{\ It is clear that }$%
	id_{\mathbb{K}^{2}}=id_{\mathbb{K}^{2}}\circ (id_{\mathbb{K}},id_{\mathbb{K}%
	})$ and then 
	\begin{equation*}
	\left\Vert id_{\mathbb{K}^{2}}\right\Vert _{BLip(\mathcal{I}_{Lip}^{1},%
		\mathcal{I}_{Lip}^{2})}\leq BLip(id_{\mathbb{K}^{2}})\left\Vert id_{\mathbb{K%
	}}\right\Vert _{\mathcal{I}_{Lip}^{1}}\left\Vert id_{\mathbb{K}}\right\Vert
	_{\mathcal{I}_{Lip}^{2}}=1.
	\end{equation*}%
	On the other hand, there are Banach spaces $G_{1},G_{2}$, $f_{j}\in \mathcal{%
		I}_{Lip}^{j}(\mathbb{K},G_{j})$ , $j=1,2$, and $R\in BLip_{0}(G_{1},G_{2};%
	\mathbb{K})$ such that $id_{\mathbb{K}^{2}}=R\circ (f_{1},f_{2})$ with 
	\begin{eqnarray*}
		1 &=&BLip(R\circ (f_{1},f_{2})) \\
		&\leq &BLip(R)Lip(f_{1})Lip(f_{2}) \\
		&\leq &BLip(R)\left\Vert f_{1}\right\Vert _{\mathcal{I}_{Lip}^{1}}\left\Vert
		f_{2}\right\Vert _{\mathcal{I}_{Lip}^{2}}.
	\end{eqnarray*}%
	Taking the infimum over all $R,$ $f_{1},$ and $f_{2}$ as above, we get $%
	1\leq \left\Vert id_{\mathbb{K}^{2}}\right\Vert _{BLip(\mathcal{I}_{Lip}^{1},%
		\mathcal{I}_{Lip}^{2})}$.
\end{example}

As a consequence of Remark \ref{rpop1}, Example \ref{remex} and the next
results we obtain that $\left( BLip(\mathcal{I}_{Lip}^{1},\mathcal{I}%
_{Lip}^{2}),\left\Vert \cdot \right\Vert _{BLip(\mathcal{I}_{Lip}^{1},%
	\mathcal{I}_{Lip}^{2})}\right) $ is a quasi-normed two-Lipschitz operator
ideal. This method of constructing an ideal of two-Lipschitz operators from
ideals of Lipschitz operators is called the \emph{factorization method}.

\begin{proposition}
	If $\mathcal{I}_{Lip}^{1},\mathcal{I}_{Lip}^{2}$ are normed Lipschitz
	operator ideals, then \newline 
		 $BLip(\mathcal{I}_{Lip}^{1},\mathcal{I}
	_{Lip}^{2})(X_{1},X_{2};E)$ is a linear subspace of $BLip_{0}(X_{1},X_{2};E)$
	and \newline  
		$\left\Vert \cdot \right\Vert _{BLip(\mathcal{I}_{Lip}^{1},\mathcal{I}
		_{Lip}^{2})}$ is a quasi-norm on $BLip(\mathcal{I}_{Lip}^{1},\mathcal{I}%
	_{Lip}^{2})(X_{1},X_{2};E)$.
\end{proposition}

\begin{proof}
	It is clear that $0\in BLip(\mathcal{I}_{Lip}^{1},\mathcal{I}%
	_{Lip}^{2})(X_{1},X_{2};E)$. By inequality (\ref{pop2}), if $\left\Vert T\right\Vert
	_{BLip(\mathcal{I}_{Lip}^{1},\mathcal{I}_{Lip}^{2})}=0$ then $T\equiv 0.$
	Let $T,T^{\prime }\in BLip(\mathcal{I}_{Lip}^{1},\mathcal{I}%
	_{Lip}^{2})(X_{1},X_{2};E).$ Then there are Banach spaces $%
	G_{11},G_{12},G_{21},G_{22}$, $f_{j}\in \mathcal{I}_{Lip}^{j}(X_{j},G_{1j})$
	, $f_{j}^{\prime }\in \mathcal{I}_{Lip}^{j}(X_{j},G_{2j})$ and two-Lipschitz
	operators $R\in BLip_{0}(G_{11},G_{12};E)$, \newline 
	$R^{\prime }\in
	BLip_{0}(G_{21},G_{22};E)$ such that $T=R\circ (f_{1},f_{2})$ and $T^{\prime
	}=R^{\prime }\circ (f_{1}^{\prime },f_{2}^{\prime })$. For $j=1,2$ consider
	the linear operators $i_{1j}:G_{1j}\longrightarrow G_{1j}\times G_{2j}$, $%
	i_{2j}:G_{2j}\longrightarrow G_{1j}\times G_{2j}$ and the mapping $%
	u_{j}:X_{j}\longrightarrow G_{1j}\times G_{2j}$ defined by $i_{1j}(x,0)=x$, $%
	i_{2j}(0,y)=y$ and $u_{j}(x):=i_{1j}\circ f_{j}(x)+i_{2j}\circ f_{j}^{\prime
	}(x)$. By using simple calculation, we prove that $u_{j}\in
	Lip_{0}(X_{j},G_{1j}\times G_{2j})$ and $Lip(u_{j})\leq
	Lip(f_{j})Lip(f_{j}^{\prime })$, $j=1,2$. Now, let us define the
	two-Lipschitz operator $B:(G_{11}\times G_{21})\times (G_{12}\times
	G_{22})\longrightarrow E$ by $B:=R\circ (\pi _{11},\pi _{12})+R^{\prime
	}\circ (\pi _{21},\pi _{22}),$ where $\pi _{1j}:G_{1j}\times
	G_{2j}\longrightarrow G_{1j}$ and $\pi _{2j}:G_{1j}\times
	G_{2j}\longrightarrow G_{2j}$ are the projections mappings defined by $\pi
	_{1j}(x,y)=x$ and $\pi _{2j}(x,y)=y$. An easy calculations show that 
	\begin{equation*}
	B\circ (u_{1},u_{2})=T+T^{\prime }\in BLip(\mathcal{I}_{Lip}^{1},\mathcal{I}%
	_{Lip}^{2})(X_{1},X_{2};E).
	\end{equation*}%
	For every $\varepsilon >0$ we can choose $R$, $R^{\prime }$, $f_{j}$, $%
	f_{j}^{\prime }$ satisfies 
	\begin{equation*}
	BLip(R)\left\Vert f_{1}\right\Vert _{\mathcal{I}_{Lip}^{1}}\left\Vert
	f_{2}\right\Vert _{\mathcal{I}_{Lip}^{2}}\leq (1+\varepsilon )\left\Vert
	T\right\Vert _{BLip(\mathcal{I}_{Lip}^{1},\mathcal{I}_{Lip}^{2})},
	\end{equation*}%
	and 
	\begin{equation*}
	BLip(R^{\prime })\left\Vert f_{1}^{\prime }\right\Vert _{\mathcal{I}%
		_{Lip}^{1}}\left\Vert f_{2}^{\prime }\right\Vert _{\mathcal{I}%
		_{Lip}^{2}}\leq (1+\varepsilon )\left\Vert T^{\prime }\right\Vert _{BLip(%
		\mathcal{I}_{Lip}^{1},\mathcal{I}_{Lip}^{2})}.
	\end{equation*}%
	Now if we put 
	\begin{equation*}
	g_{j}=\frac{f_{j}}{\left\Vert f_{j}\right\Vert _{\mathcal{I}_{Lip}^{j}}}%
	\left\Vert T\right\Vert _{BLip(\mathcal{I}_{Lip}^{1},\mathcal{I}%
		_{Lip}^{2})}^{\frac{1}{2}}
	\end{equation*}%
	and 
	\begin{equation*}
	S=\frac{\left\Vert f_{1}\right\Vert _{\mathcal{I}_{Lip}^{1}}\left\Vert
		f_{2}\right\Vert _{\mathcal{I}_{Lip}^{2}}}{\left\Vert T\right\Vert _{BLip(%
			\mathcal{I}_{Lip}^{1},\mathcal{I}_{Lip}^{2})}}R,
	\end{equation*}%
	we obtain $T=S\circ (g_{1},g_{2})$ with $BLip(S)\leq (1+\varepsilon )$ and $%
	\left\Vert g_{j}\right\Vert _{\mathcal{I}_{Lip}^{j}}\leq \left\Vert
	T\right\Vert _{BLip(\mathcal{I}_{Lip}^{1},\mathcal{I}_{Lip}^{2})}^{\frac{1}{2%
	}}$ . By a similar argument, we get $T^{\prime }=S^{\prime }\circ
	(g_{1}^{\prime },g_{2}^{\prime })$ with $BLip(S^{\prime })\leq
	(1+\varepsilon )$ and $\left\Vert g_{j}^{\prime }\right\Vert _{\mathcal{I}%
		_{Lip}^{j}}\leq \left\Vert T^{\prime }\right\Vert _{BLip(\mathcal{I}%
		_{Lip}^{1},\mathcal{I}_{Lip}^{2})}^{\frac{1}{2}}$. In this case 
	\begin{equation*}
	\left\Vert u_{j}\right\Vert _{\mathcal{I}_{Lip}^{j}}\leq \left\Vert
	i_{1j}\circ g_{j}\right\Vert _{\mathcal{I}_{Lip}^{j}}+\left\Vert i_{2j}\circ
	g_{j}^{\prime }\right\Vert _{\mathcal{I}_{Lip}^{j}}\leq \left\Vert
	g_{j}\right\Vert _{\mathcal{I}_{Lip}^{j}}+\left\Vert g_{j}^{\prime
	}\right\Vert _{\mathcal{I}_{Lip}^{j}}.
	\end{equation*}%
	Although a standard calculation gives $BLip(B)\leq (1+\varepsilon )$ and
	then 
	\begin{center}
		$\left\Vert T+T^{\prime }\right\Vert _{BLip(\mathcal{I}_{Lip}^{1},\mathcal{I}%
			_{Lip}^{2})}=$
		
		$=\left\Vert B\circ (u_{1},u_{2})\right\Vert _{BLip(\mathcal{I}_{Lip}^{1},%
			\mathcal{I}_{Lip}^{2})}$
		
		$\leq (1+\varepsilon )\left( \left\Vert g_{1}\right\Vert _{\mathcal{I}%
			_{Lip}^{1}}+\left\Vert g_{1}^{\prime }\right\Vert _{\mathcal{I}%
			_{Lip}^{1}}\right) \left( \left\Vert g_{2}\right\Vert _{\mathcal{I}%
			_{Lip}^{2}}+\left\Vert g_{2}^{\prime }\right\Vert _{\mathcal{I}%
			_{Lip}^{2}}\right) $
		
		$\leq (1+\varepsilon )\left( \left\Vert T\right\Vert _{BLip(\mathcal{I}%
			_{Lip}^{1},\mathcal{I}_{Lip}^{2})}^{\frac{1}{2}}+\left\Vert T^{\prime
		}\right\Vert _{BLip(\mathcal{I}_{Lip}^{1},\mathcal{I}_{Lip}^{2})}^{\frac{1}{2%
		}}\right) ^{2}$
		
		$\leq (1+\varepsilon )\left( 2\left( \left\Vert T\right\Vert _{BLip(\mathcal{%
				I}_{Lip}^{1},\mathcal{I}_{Lip}^{2})}+\left\Vert T^{\prime }\right\Vert
		_{BLip(\mathcal{I}_{Lip}^{1},\mathcal{I}_{Lip}^{2})}\right) ^{\frac{1}{2}%
		}\right) ^{2}$
		
		$=4(1+\varepsilon )\left( \left\Vert T\right\Vert _{BLip(\mathcal{I}%
			_{Lip}^{1},\mathcal{I}_{Lip}^{2})}+\left\Vert T^{\prime }\right\Vert _{BLip(%
			\mathcal{I}_{Lip}^{1},\mathcal{I}_{Lip}^{2})}\right) .$
	\end{center}%
	Since $\varepsilon $ is arbitrary it follows that 
	\begin{equation*}
	\left\Vert T+T^{\prime }\right\Vert _{BLip(\mathcal{I}_{Lip}^{1},\mathcal{I}%
		_{Lip}^{2})}\leq 4\left( \left\Vert T\right\Vert _{BLip(\mathcal{I}%
		_{Lip}^{1},\mathcal{I}_{Lip}^{2})}+\left\Vert T^{\prime }\right\Vert _{BLip(%
		\mathcal{I}_{Lip}^{1},\mathcal{I}_{Lip}^{2})}\right) .
	\end{equation*}%
	It is easy to check that $\lambda T\in BLip(\mathcal{I}_{Lip}^{1},\mathcal{I}%
	_{Lip}^{2})(X_{1},X_{2};E)$ for every $\lambda \in \mathbb{K}$ and $T\in
	BLip(\mathcal{I}_{Lip}^{1},\mathcal{I}_{Lip}^{2})(X_{1},X_{2};E)$.
\end{proof}

The next proposition follows directly from the definitions and we omit the
proof.

\begin{proposition}
	(Ideal property). \textit{Let} $f\in Lip_{0}(Z,X_{1})$, $g\in
	Lip_{0}(W,X_{2})$ and $u\in \mathcal{L}(E,F)$. \textit{If} $T\in BLip(%
	\mathcal{I}_{Lip}^{1},\mathcal{I}_{Lip}^{2})(X_{1},X_{2};E)$\textit{, then }%
	the composition $u\circ T\circ (f,g)$ \textit{is }in $BLip(\mathcal{I}%
	_{Lip}^{1},\mathcal{I}_{Lip}^{2})(Z,W;F)$\textit{\ and} 
	\begin{equation*}
	\Vert u\circ T\circ (f,g)\Vert _{BLip(\mathcal{I}_{Lip}^{1},\mathcal{I}%
		_{Lip}^{2})}\leq \left\Vert u\right\Vert \Vert T\Vert _{BLip(\mathcal{I}%
		_{Lip}^{1},\mathcal{I}_{Lip}^{2})}Lip(f)Lip(g).
	\end{equation*}
\end{proposition}

The aim of the next proposition is to show the injectivity of the
two-Lipschitz operator ideal constructed by the factorization method. In
order to prove this result we need the following lemma. The proof of the
lemma is similar to the Lipschitz case (see \cite[Proposition 1.6]{Weav2}).

\begin{lemma}
	Let $X$ and $Y$ be complete pointed metric spaces and let $E$ be a Banach
	space. If $X_{0}$ and $Y_{0}$ are dense subsets of $X$ and $Y$ respectively,
	then every two-Lipschitz operator $T\in BLip_{0}\left( X_{0},Y_{0};E\right) $
	has a unique two-Lipschitz operator extension $\widetilde{T}\in
	BLip_{0}\left( X,Y;E\right) $ with $BLip(T)=BLip(\widetilde{T})$.
\end{lemma}

\begin{proposition}
	Let $\mathcal{I}_{Lip}^{1},\mathcal{I}_{Lip}^{2}$ be injective Lipschitz
	operator ideals, then the two-Lipschitz operator ideal $BLip(\mathcal{I}%
	_{Lip}^{1},\mathcal{I}_{Lip}^{2})$ is also injective.
\end{proposition}

\begin{proof}
	If $T\in (BLip(\mathcal{I}_{Lip}^{1},\mathcal{I}_{Lip}^{2}))^{inj}\left(
	X_{1},X_{2};E\right) $, then there are Banach spaces $G_{1},G_{2}$,
	Lipschitz operator $f_{j}\in \mathcal{I}_{Lip}^{j}(X_{j},G_{j})$, $j=1,2$,
	and a two-Lipschitz operator $R\in BLip_{0}\left( G_{1},G_{2};\ell _{\infty
	}(B_{E^{\ast }})\right) \ $such that $J_{E}\circ T=R\circ (f_{1},f_{2})$.
	Consider the Lipschitz operator $g_{j}:X_{j}\longrightarrow \overline{%
		span\left\{ f_{j}(X_{j})\right\} }$ defined by $g_{j}(x_{j})=f_{j}(x_{j})$
	and the canonical inclusion $i_{j}:\overline{span\left\{
		f_{j}(X_{j})\right\} }\hookrightarrow G_{j},$ $j=1,2$. Then, $i_{j}\circ
	g_{j}=f_{j}\in \mathcal{I}_{Lip}^{j}(X_{j},G_{j})$ and by the injectivity of 
	$\mathcal{I}_{Lip}^{j}$ we get $g_{j}\in \mathcal{I}_{Lip}^{j}\left( X_{j},%
	\overline{span\left\{ f_{j}(X_{j})\right\} }\right) ,$ $j=1,2$. Now, define
	the two-Lipschitz operator $$S\in \mathcal{L}\left( span\left\{
	f_{1}(X_{1})\right\} ,span\left\{ f_{2}(X_{2})\right\} ;E\right) $$ by $
	S\left( f_{1}(x_{1}),f_{2}(x_{2})\right) =T(x_{1},x_{2})$ were extended to
	two-Lipschitz operator $\widetilde{S}$ from $\overline{span\left\{
		f_{1}(X_{1})\right\} }\times \overline{span\left\{ f_{2}(X_{2})\right\} }$
	into $E$. So $T=\widetilde{S}\circ \left( g_{1},g_{2}\right) \in BLip(%
	\mathcal{I}_{Lip}^{1},\mathcal{I}_{Lip}^{2})$.
\end{proof}

\subsection{Strong factorization method}

In the above method, if we consider a bilinear operator $R\in \mathcal{B}
(G_{1},G_{2};E)$ between Banach spaces instead of the two-Lipschitz operator 
$R$, we get another ideal of two-Lipschitz operators constructed from the
Lipschitz ideals $\mathcal{I}_{Lip1}$ and$\ \mathcal{I}_{Lip2},$ denoted by $
\mathcal{B}(\mathcal{I}_{Lip1},\mathcal{I}_{Lip2}).$ This method is called
the \emph{strong factorization method}. If $\mathcal{I}_{Lip}^{1},\mathcal{I}
_{Lip}^{2}$ are normed Lipschitz operator ideals and $T\in \mathcal{B}(
\mathcal{I}_{Lip}^{1},\mathcal{I}_{Lip}^{2})(X_{1},X_{2};E)$ we define 
\begin{equation*}
\left\Vert T\right\Vert _{\mathcal{B}(\mathcal{I}_{Lip}^{1},\mathcal{I}
	_{Lip}^{2})}=\inf \left\Vert R\right\Vert \left\Vert f_{1}\right\Vert _{
	\mathcal{I}_{Lip}^{1}}\left\Vert f_{2}\right\Vert _{\mathcal{I}_{Lip}^{2}},
\end{equation*}
where the infimum is taken over all possible factorizations $T=R\circ
(f_{1},f_{2})$.

The proof of the following proposition is similar to their factorization
method analogues.

\begin{proposition}
	If $\mathcal{I}_{Lip}^{1},\mathcal{I}_{Lip}^{2}$ are normed Lipschitz
	operator ideals, then \newline $\mathcal{B}(\mathcal{I}_{Lip}^{1},\mathcal{I}
	_{Lip}^{2})$ is a quasi-normed two-Lipschitz operator ideal with the
	quasi-norm $\left\Vert \cdot \right\Vert _{\mathcal{B}(\mathcal{I}_{Lip}^{1},
		\mathcal{I}_{Lip}^{2})}$.
\end{proposition}
Now we investigate the symmetry of the strong factorization ideals.

\begin{theorem}
	Let $\mathcal{I}_{Lip}^{1}$ and $\mathcal{I}_{Lip}^{2}$ be Lipschitz
	operator ideals. Then the following assertions are equivalent.
	
	\begin{enumerate}
		\item[(1)] The two-Lipschitz operator ideal $\mathcal{B}(\mathcal{I}
		_{Lip}^{1},\mathcal{I}_{Lip}^{2})$ is symmetric.
		
		\item[(2)] $\mathcal{B}(\mathcal{I}_{Lip}^{1},\mathcal{I}_{Lip}^{2})=
		\mathcal{B}(\mathcal{I}_{Lip}^{2},\mathcal{I}_{Lip}^{1}).$
		
		\item[(3)] $\mathcal{I}_{Lip}^{1}=\mathcal{I}_{Lip}^{2}.$
	\end{enumerate}
\end{theorem}

\begin{proof}
	$(2)\Longrightarrow (1)$ Let $T\in \mathcal{B}(\mathcal{I}_{Lip}^{1},
	\mathcal{I}_{Lip}^{2})(X,X;E)$. Then there are Banach spaces $G_{1},G_{2}$,
	Lipschitz operator $f_{j}\in \mathcal{I}_{Lip}^{j}(X_{j},G_{j})$ and
	bilinear operator $R\in \mathcal{B}(G_{1},G_{2};E)$ such that $T=R\circ
	(f_{1},f_{2})$. Consider the mappings $T^{\prime }\in BLip_{0}(X,X;E)$, $
	R^{\prime }\in \mathcal{B}(G_{2},G_{1};E)$ defined by $T^{\prime
	}(x_{1},x_{2}):=T(x_{2},x_{1})$ and $R^{\prime }(g_{2},g_{1}):=R(g_{1},g_{2})
	$ and notice that we have $T^{\prime }=R^{\prime }\circ (f_{1},f_{2})$. This
	implies that 
	\begin{equation*}
	T^{s}=\frac{1}{2}(T+T^{\prime })\in \mathcal{B}(\mathcal{I}_{Lip}^{1},
	\mathcal{I}_{Lip}^{2})(X,X;E).
	\end{equation*}
	$(1)\Longrightarrow (3)$ Let $\phi \in \mathcal{I}_{\mathcal{I}
		_{Lip}^{1}}(X,E)$. Fix $f\in X^{\#}$ and $a\in X$ with $f(a)=1$. Define $
	T\in BLip_{0}(X,X;E)$, $R\in \mathcal{L}(\mathbb{K},E;E)$ by $
	T(x_{1},x_{2}):=\phi (x_{1})f(x_{2})$ and $R(\alpha ,e):=\alpha e$. It
	follows that $R\circ (\phi ,f)=T\in \mathcal{B}(\mathcal{I}_{Lip}^{1},
	\mathcal{I}_{Lip}^{2})(X,X;E).$ By the hypothesis we have that $T^{s}\in 
	\mathcal{B}(\mathcal{I}_{Lip}^{1},\mathcal{I}_{Lip}^{2})(X,X;E)$, and then $
	T^{s}=S\circ (f_{1},f_{2})$ with $f_{j}\in \mathcal{I}_{Lip}^{j}(X_{j},G_{j})
	$ and $S\in \mathcal{B}(G_{1},G_{2};E)$. Consider $S_{L}\in \mathcal{L}(G_{1}
	\widehat{\otimes }_{\pi }G_{2},E)$ the linearization of the bilinear
	operator $S$ that is $S_{L}(g_{1}\otimes g_{2})=S(g_{1},g_{2})$ and the
	linear operator $\varphi \in \mathcal{B}(G_{2},G_{1}\widehat{\otimes }_{\pi
	}G_{2})$ by $\varphi (g_{2})=f_{1}(a)\otimes g_{2}$. Some direct computations
	show that 
	\begin{equation*}
	\phi =2S_{L}\circ \varphi \circ f_{2}-\phi (a)f\in \mathcal{I}
	_{Lip}^{2}(X,E).
	\end{equation*}
	Therefore $\mathcal{I}_{\mathcal{I}_{Lip}^{1}}\subset \mathcal{I}_{\mathcal{I
		}_{Lip}^{2}}$. Using a similar argument as above we get the reverse
	inclusion.
	
	The direction $(3)\Longrightarrow (2)$ is obvious.
\end{proof}

In the next result we will see that the Lipschitzization and strong factorization methods provide an inclusion relationship.

\begin{proposition}
	If $\mathcal{I}_{Lip}^{1},\mathcal{I}_{Lip}^{2}$ are normed Lipschitz
	operator ideals, then 
	\begin{equation*}
	\mathcal{B}(\mathcal{I}_{Lip}^{1},\mathcal{I}_{Lip}^{2})\subset \left[ 
	\mathcal{I}_{Lip}^{1},\mathcal{I}_{Lip}^{2}\right] ,
	\end{equation*}
	with $\left\Vert \cdot \right\Vert _{\left[ \mathcal{I}_{Lip}^{1},\mathcal{I}
		_{Lip}^{2}\right] }\leq \left\Vert \cdot \right\Vert _{\mathcal{B}(\mathcal{I
		}_{Lip}^{1},\mathcal{I}_{Lip}^{2})}$.
\end{proposition}

\begin{proof}
	If $T\in \mathcal{B}(\mathcal{I}_{Lip}^{1},\mathcal{I}%
	_{Lip}^{2})(X_{1},X_{2};E)$, for $\varepsilon >0$ choose Banach spaces $
	G_{1},G_{2}$, bilinear operator $R:G_{1}\times G_{2}\longrightarrow E$ and $
	f_{j}\in \mathcal{I}_{Lip}^{j}(X_{j},G_{j})$ , $j=1,2$ such that $T=R\circ
	(f_{1},f_{2})$ with $\left\Vert R\right\Vert \left\Vert f_{1}\right\Vert _{
		\mathcal{I}_{Lip}^{1}}\left\Vert f_{2}\right\Vert _{\mathcal{I}
		_{Lip}^{2}}\leq \varepsilon +\left\Vert T\right\Vert _{\mathcal{B}(\mathcal{I}_{Lip}^{1},\mathcal{I}_{Lip}^{2})}$. Consider the linear operators, 
	\begin{equation*}
	\begin{array}{cc}
	u_{1}:G_{1}\longrightarrow Lip_{0}(X_{2},E), & u_{2}:G_{2}\longrightarrow
	Lip_{0}(X_{2},E),
	\end{array}
	\end{equation*}
	defined by $u_{1}(g_{1})(x_{2})=R(g_{1},f_{2}(x_{2}))$ and $
	u_{2}(g_{2})(x_{1})=R(f_{1}(x_{1}),g_{2})$ for every $g_{1}\in G_{1},$ $
	g_{2}\in G_{2},$ $x_{1}\in X_{1},$ $x_{2}\in X_{2}.$ A simple calculation
	gives $u_{1}\circ f_{1}=\left( R\circ (f_{1},f_{2})\right) _{1}$ and $
	\left\Vert u_{1}\right\Vert \leq \left\Vert R\right\Vert Lip(f_{2})$. Which
	means that $T_{1}\in \mathcal{I}_{Lip}^{1}(X_{1},Lip_{0}(X_{2},E))$ and 
	\begin{equation*}
	\left\Vert T_{1}\right\Vert _{\mathcal{I}_{Lip}^{1}}\leq \left\Vert
	u_{1}\right\Vert Lip(f_{1})\leq \varepsilon +\left\Vert T\right\Vert _{
		\mathcal{B}(\mathcal{I}_{Lip}^{1},\mathcal{I}_{Lip}^{2})}.
	\end{equation*}
	Used the same argument, we obtain $T_{2}\in \mathcal{I}
	_{Lip}^{2}(X_{2},Lip_{0}(X_{1},E))$ and 
	\begin{equation*}
	\left\Vert T_{2}\right\Vert _{\mathcal{I}_{Lip}^{2}}\leq \left\Vert
	u_{2}\right\Vert Lip(f_{2})\leq \varepsilon +\left\Vert T\right\Vert _{
		\mathcal{B}(\mathcal{I}_{Lip}^{1},\mathcal{I}_{Lip}^{2})}.
	\end{equation*}
	Which implies that $T\in \left[ \mathcal{I}_{Lip1},\mathcal{I}_{Lip2}\right]
	(X_{1},X_{2};E)$ and $\left\Vert T\right\Vert _{\left[ \mathcal{I}_{Lip}^{1},
		\mathcal{I}_{Lip}^{2}\right] }\leq \left\Vert T\right\Vert _{\mathcal{B}(\mathcal{I}_{Lip}^{1},\mathcal{I}_{Lip}^{2})}.$
\end{proof}

\end{document}